\documentclass{article}
\usepackage{amsmath,amssymb,amsthm,graphicx,epsfig,float,url}
\usepackage[colorlinks=true,citecolor={Plum},linkcolor={Periwinkle}]{hyperref}
\usepackage{pdfsync}
\usepackage[usenames, dvipsnames]{xcolor}
\usepackage{tikz}
\usepackage{subfig}
\usepackage{bbm}
\usepackage{stmaryrd}
\usepackage{mathrsfs}  
\usepackage{caption}
\usepackage{float}
\usepackage{esint}
\usepackage[english]{babel}
\usepackage{dsfont}
\usepackage{graphicx}
\usepackage{grffile}
\usepackage{caption}

\usepackage[makeroom]{cancel}
\usetikzlibrary{patterns}
\newcommand{\R}{\textnormal{I\kern-0.21emR}}
\newcommand{\N}{\textnormal{I\kern-0.21emN}}

\renewcommand{\geq}{\geqslant}
\renewcommand{\leq}{\leqslant}

\def\e{{\varepsilon}}
\def\tmm{{\theta_{m,\mu}}}

\allowdisplaybreaks

\def\YYint#1#2#3{{\setbox0=\hbox{$#1{#2#3}{\iint}$}
    \vcenter{\hbox{$#2#3$}}\kern-.51\wd0}}
 
\usepackage[dvipsnames]{xcolor}

\newtheorem*{theorem*}{Theorem}

\newtheorem{theorem}{Theorem}

\newtheorem{lemma}{Lemma}

\theoremstyle{definition}\newtheorem{remark}{Remark}

\def\O{{\Omega}}
\def\n{{\nabla}}
\def\p{{\varphi}}

\usepackage{xargs}% use more than one optional parameter in a new command
 \usepackage[colorinlistoftodos,textsize=small]{todonotes}
 \newcommandx{\unsure}[2][1=]{\todo[linecolor=red,backgroundcolor=red!25,bordercolor=red,#1]{#2}}
 \newcommandx{\change}[2][1=]{\todo[linecolor=blue,backgroundcolor=blue!25,bordercolor=blue,#1]{#2}}
 \newcommandx{\info}[2][1=]{\todo[linecolor=green,backgroundcolor=green!25,bordercolor=green,#1]{#2}}
 \newcommandx{\improvement}[2][1=]{\todo[linecolor=yellow,backgroundcolor=yellow!25,bordercolor=yellow,#1]{#2}}
 
  \newcommandx{\biblio}[2][1=]{\todo[linecolor=blue,backgroundcolor=magenta!25,bordercolor=blue,#1]{#2}}

\begin{document}
\nocite{*}
\title{A fragmentation phenomenon for a non-energetic optimal control problem: \\optimisation of the total population size in logistic diffusive models}

%    Remove any unused author tags.

%    author one information
\author{Idriss Mazari\footnote{Institute of Analysis and Scientific Computing, TU Wien, Wiedner Hauptstrasse 8-10, 1040 Vienna, Austria. (\texttt{idriss.mazari@sorbonne-universite.fr})},  \quad Dom\`enec Ruiz-Balet \footnote{Chair of Computational Mathematics, Fundaci\'on Deusto, Av. de las Universidades, 24,  48007 Bilbao, Basque Country, Spain} \footnote{ Departamento de Matem\'eticas, Universidad Aut\'onoma de Madrid, 28049 Madrid, Spain}}
\date{\today}
%\address{}
%\curraddr{}
%\email{idriss.mazari@upmc.fr}
%\thanks{}

\maketitle

\begin{abstract}
Following the recent works \cite{DeAngelis2020,InoueKuto,MNP,NagaharaYanagida,Zhang2017}, we investigate the problem of optimising the total population size for logistic diffusive models with respect to resources distributions. Using the spatially heterogeneous Fisher-KPP equation, we obtain a surprising fragmentation phenomenon: depending on the scale of diffusivity (i.e the dispersal rate), it is better to either concentrate or fragment resources. Our main result is that, the smaller the dispersal rate of the species in the domain, the more optimal resources distributions tend to oscillate. This is in sharp contrast with other criteria in population dynamics, such as the classical problem of optimising the survival ability of a species, where concentrating resources is always favourable, regardless of the diffusivity. Our study is completed by numerous numerical simulations that confirm our results. 

\end{abstract}

\noindent\textbf{Keywords:} diffusive logistic equation, optimal control, shape optimization.

\medskip

\noindent\textbf{AMS classification:} 35Q92,49J99,34B15.

%\tableofcontents
%INTRODUCTION

\section{Introduction}
\subsection{Scope of this article: fragmentation and concentration for spatial ecology}
In this article, we study a problem of great relevance in the field of spatial ecology. Namely, considering a species dispersing in a domain where some resources are available:
\begin{center}
\emph{How should we spread resources so as to maximise the total population size at equilibrium?}\end{center}
Here, we focus on a fine qualitative analysis of this question and emphasise the crucial role of the characteristic diffusion rate of the population (or, equivalently, of the size of the domain).

Regarding mathematical biology, spatially heterogeneous models are of paramount importance, as acknowledged, for instance, in \cite{GC03}. Natural questions arise when considering such models: one may for instance think of spatially heterogeneous systems of reaction-diffusion equations, in which case a relevant question is that of existence and stability of (non-trivial) equilibria (see for instance \cite{HN,HENIII,HeNi,Mazari}).

 Here, we focus on single-species models, in which case two problems have drawn a lot of attention from both the mathematical and the mathematical biology communities: the problem of optimal survival ability, and the problem of optimising the total population size. We expand on  bibliographical references in Subsection \ref{Se:Bibl} of this Introduction, but let us stress the following fact: while the optimisation of the survival ability with respect to resources distribution is fairly well-understood (in terms of qualitative analysis, see for instance \cite{BHR,KaoLouYanagida,LamboleyLaurainNadinPrivat}), the problem of the total population size, which has been the subjects of several recent articles (we refer for instance to \cite{BaiHeLi,DeAngelis2020,InoueKuto,LouNagaharaYanagida,MNP,NagaharaYanagida,Zhang2017}) is still elusive when considered from a qualitative point of view. For instance, for the optimal survival ability, the following paradigm has been established:
\begin{center}\begin{emph}{
Concentration of resources favors survival ability.}
\end{emph}\end{center}
This was first observed  in \cite{ShigesadaKawaski}, and given a proper mathematical analysis in \cite{BHR}, in terms of rearrangements.
One of the other conclusions of \cite{BHR} is that heterogeneity is favorable to survival ability: under natural assumptions (made precise in Section \ref{Se:Bibl} through the definition of the admissible class, Equation \eqref{Eq:Ad}), in order to maximise the survival ability of a population, one should work with patch-models. Here, this means the following: provided the population evolves in $\O$ and the resources distributions $m:\O\to \R$ satisfy pointwise  ($0\leq m\leq 1$) and integral ($\int_\O m\leq C$) bounds, the optimal resources distribution for survival ability $m^*$ satisfies $\O=\{m^*=0\}\sqcup\{m^*=1\}$. Such results generally do not depend on the dispersal rate: regardless of this characteristic speed, resources distributions should be concentrated if we want to optimise the survival ability.

The problem of optimising the total population size, on the other hand, is much more complicated to tackle at the mathematical level. One of the main questions that have been investigated is the influence of diffusion on the population size criterion (which in some models favours the total population size, see \cite{LouInfluence}), and we refer to \cite{Zhang2017}, as well as the recent survey \cite{DeAngelis2020} for a biological perspective on this question. In these two last references, the following question is also presented: can the total population size exceed the total amount of resources? This question, for the model we are going to consider, has been solved in dimension 1 in \cite{BaiHeLi} and, in dimension $n\geq 2$, in the recent preprint \cite{InoueKuto}. In all of these papers, the dispersal rate plays a crucial role in the analysis.

Regarding qualitative properties, as will be explained further in Section \ref{Se:Bibl}, very few things are known. The relevance of patch-models for this optimisation problem has been investigated in \cite{MNP} and \cite{NagaharaYanagida}, but, so far, the only qualitative results can be found in \cite{MNP}: for large dispersal rates, concentrating resources favours the total population size while, for small diffusivities, fragmentation (i.e. scattering resources across the domain) may be better.

The goal of this article is to give a complete treatment of the case of small dispersal rates for the logistic-diffusive Fisher-KPP equation, and our main result, Theorem \ref{Th:Frag}, may be interpreted as follows:
\begin{center}
\begin{emph}{
To maximise the total population size, the smaller the diffusivity, the more one should fragment resources.}
\end{emph}
\end{center}

From a calculus of variations (or optimal control) perspective, our article's main innovation is that it gives a qualitative analysis of a non-energetic optimisation problems. Such problems are notoriously hard to analyse, given that their structure prohibits using classical tools (e.g. rearrangements, symmetrisation) and that the analysis of optimality conditions is very tricky. Here, we propose an approach relying on strong non-monotonicity properties of the functional that is to be optimised.

Finally, we provide several numerical experiments that validate our results. 

This article is organised as follows: In Section \ref{Se:Bibl}, we present the model and the variational problem under consideration,  and recall the several qualitative properties available in the literature. In Section \ref{TH}, we state our main result.  Its proof is given in Section \ref{Proof}. In Section \ref{Num}, we give several numerical simulations to illustrate Theorem \ref{Th:Frag}. Finally, we present concluding comments and open questions in the Conclusion.

\subsection{Setting and bibliographical references}\label{Se:Bibl}

We are working here with the spatially heterogeneous Fisher-KPP equation (which originated in the seminal \cite{Fisher,KPP}). Let us  consider, in dimension $n$, the box 
$$\O=(0;1)^n,$$
which will serve as our domain. We could consider more general boxes $\O=\prod_{i=1}^n(0;a_i)$, but the results and proofs would be the same. We consider a positive parameter $\mu>0$, which will be referred to as dispersal rate, or diffusivity. To model the spatial heterogeneity, we use resources distributions, i.e. functions $m:\O\to \R$. Finally, we take into account an intra-specific, non-linear reaction term from the classical logistic equation. This gives the following equation: assuming the population density $\tmm$ has reached an equilibrium, it solves

\begin{equation}\label{LDE}
\begin{cases}
\mu \Delta \tmm+\tmm\left(m-\tmm\right)=0\text{ in }\O,
\\
\\\frac{\partial \tmm}{\partial \nu}=0\text{ on }\partial \O,
\\
\\\tmm>0\text{ in }\O.
\end{cases}
\end{equation}
For Equation \eqref{LDE} to have a solution, one must restrict the class of resources distributions. If we assume 
$$m\in L^\infty(\O)\,, \int_\O m>0,$$ then \cite{BHR,CantrellCosner1,MR1105497} guarantee the existence, uniqueness and stability of a  solution to \eqref{LDE}. 

We introduce the functional 
$$F:(m,\mu)\mapsto \fint_\O \tmm,$$ where, for any function $\p\in L^1(\O)$, the notation $\fint_\O \p$ stands for 
$$\fint_\O \p:=\frac1{|\O|}\int_\O \p,$$ and we consider the optimisation problem
$$\sup_{m\,, \fint_\O m>0} F(m,\mu)=\fint_\O \tmm.$$
This problem is ill-posed without further constraints on $m$. Two natural constraints can be set, a pointwise ($L^\infty$) constraint, and a $L^1$ constraint, which leads to introducing the admissible class:

\begin{equation}\label{Eq:Ad}\mathcal M(\O):=\left\{m\in L^\infty(\O)\,, 0\leq m\leq \kappa\,, \fint m=m_0\right\}\end{equation} where $\kappa\,, m_0>0$ are two positive constants (we require $m_0<\kappa$ to ensure that $\mathcal M(\O)\neq \emptyset$).
This admissible class was proposed in \cite{Lou2008} and used, for instance, in \cite{MNP,NagaharaYanagida}.

The optimisation problem under consideration is 
\begin{equation}\label{PV}\tag{$P_\mu$}
\fbox{$\displaystyle \max_{m\in \mathcal M(\O)}\fint_\O \tmm.$}\end{equation}
The direct method of the calculus of variations yields in a straightforward way the existence of a solution $m_\mu^*\in \mathcal M(\O)$.

\paragraph{A remark on the constraints}
We would like to stress the importance of the pointwise constraint $0\leq m\leq \kappa$.
As mentioned in the first part of this Introduction, a natural question was that of knowing whether the total population size could exceed the total amount of resources, see \cite{DeAngelis2020,Zhang2017}. In other words, what can be said about the ratio
$$E(m):=\sup_{\mu>0}\frac{\fint_\O \tmm}{\fint_\O m},$$ where $m$ satisfies $m\geq 0$, $m\neq 0$?
It follows from \cite{LouInfluence} that 
$$E(m)\geq 1.$$

 In the one-dimensional case, Bai, He and Li proved, in \cite{BaiHeLi} that 
$$E(m)\leq 3$$ and that this bound is not attained for any $m$. 

In the $n$-dimensional case, $n\geq 2$, Inoue and Kuto  proved \cite{InoueKuto}
$$\sup_{m\geq 0\,, m\neq 0}E(m)=+\infty.$$
\normalsize

\paragraph{Upper and lower bound on \eqref{PV}} In \cite{LouInfluence}, it is established that, for any $\mu>0$, $m\equiv m_0$ was a global strict minimizer of $F(\cdot,\mu)$ in $\mathcal M(\O)$. The maximum principle implies that, for any $\mu>0$ and any $m\in \mathcal M(\O)$, we have $\tmm\leq \kappa$, so that we get the following upper and lower bounds on the criterion:
$$\forall \mu>0\,, \forall m\in \mathcal M(\O)\,, 0<m_0\leq F(m,\mu)\leq \kappa.$$
This upper bound is very crude but, to the best of our knowledge, is the best one for all diffusivities. It is possible to refine it in several cases. In \cite{MazariNadinPrivat} it is for instance proved that the value of \eqref{PV} converges to $m_0$ as $\mu\to \infty$. In sharp contrast, when $\mu \to 0$,  Remark \ref{Rem} of the present article seems to indicate that the sharpest upper bound, i.e. the value of \eqref{PV}, will actually converge to $\sup_{\mu>0\,, m\in \mathcal M(\O)}F(m,\mu)$.  \color{black}

\paragraph{Qualitative properties for \eqref{PV}}
One of the main features of problems such as \eqref{PV} is the bang-bang property: denoting by $m_\mu^*$ a maximiser for \eqref{PV}, is it true that there exists a set $E_\mu^*$ such that $m_\mu^*=\kappa \mathds 1_{E_\mu^*}$? Such characteristic functions are called bang-bang functions. This property is of paramount importance in optimisation and, from a mathematical biology point of view, corroborates the relevance of patch-models, see \cite{CantrellCosner}.
 
Let us briefly sum up the main conclusions of \cite{MNP,NagaharaYanagida}, which contain the most up to date qualitative informations of that sort about \eqref{PV}:
\begin{enumerate}
\item A bang-bang property is proved in \cite{NagaharaYanagida}: if the set $\{0<m<\kappa\}$ contains an open ball, then $m$ is not a solution of \eqref{PV}. Here, a regularity assumption is thus needed.
\item In \cite{MNP}, it is proved that the bang-bang property holds for all large enough diffusivities.
\item It is furthermore proved, also in \cite{MNP} that:
\begin{enumerate}
\item In the one-dimensional case $\O=(0;1)$, there exists $\mu_1>0$ such that, for every $\mu\geq \mu_1$, the unique solutions of \eqref{PV} are $$m^*:=\kappa \mathds 1_{(0;\ell)}\text{ and }m_*:=\kappa \mathds 1_{(1-\ell;\ell)},$$ with $\kappa\ell=m_0$.
We note that these are also optimal configurations for survival ability (see \cite{BHR,LamboleyLaurainNadinPrivat}).
\item In the $2$-dimensional case $\O=(0;1)\times (0;1)$, concentration holds for large diffusivities in the following sense: any sequence $\{m_\mu^*\}_{\mu\to \infty}$ of solutions of \eqref{PV} converges, up to a subsequence, to a bang-bang function $m_\infty^*=\kappa \mathds 1_{E_\infty^*}$ which is non-increasing in every direction; in other words, $x\mapsto m(x,y)$ (resp. $y\mapsto m(x,y)$) is non-increasing for almost every $y$ (resp. non-increasing for almost every $x$)
\end{enumerate}
\item In \cite{MNP} it is proved that, for small enough diffusivities, fragmentation may be better in the following sense: two crenels are better than one crenel.\end{enumerate}

\paragraph{Qualitative properties for a discretized version of \eqref{PV}}   In the recent \cite{LouNagaharaYanagida}, a spatially discretized version of Equation \eqref{LDE} and of the optimization problem \eqref{PV} is considered in dimension 1. In this so-called \textquotedblleft patchy environment model \textquotedblright, they obtain a complete characterization of maximizers for certain classes of parameters and, most notably, establish the periodicity of optimal resources distributions for certain values of these parameters. 

\color{black}

\color{white}ajikbouob\color{black}

As already mentioned, our goal is to prove a strong fragmentation phenomenon for small diffusivities. A way to formalise this fragmentation would be to restrict ourselves to looking for bang-bang solutions of \eqref{PV}, i.e. solutions of the form $m_\mu^*=\kappa\mathds1_{E_\mu^*}$ and to prove that 
$$Per(E_\mu^*)\underset{\mu\to 0}\rightarrow +\infty,$$ where $Per(E_\mu^*)$ is the Cacciopoli perimeter of the set $E_\mu^*$:
$$
Per(E_\mu^*)=\sup\left\{\int_\Omega \mathds 1_E \mathrm{div}{\phi}\,,  {\phi}\in C_c^1(\Omega,\R^n),\ \|{\phi}\|_{L^\infty(\Omega)}\le 1\right\}.$$

However, the problem 
$$\sup_{m\in \mathcal M(\O)\,, \text{ $m$ bang-bang}}F(m,\mu)$$ does not necessarily have a solution, as remarked above.

\paragraph{Mathematical formulation of fragmentation}
To quantify the perimeter or the regularity of the optimal resources distribution, we  introduce, for a fixed $M>0$, the class
\begin{equation}\label{Eq:AdM}
\mathcal M_M(\O):=\left\{m\in \mathcal M(\O)\,, \Vert m\Vert_{BV(\O)}\leq M\right\}.\end{equation}
Here, the $BV(\O)$-norm refers to the bounded-variations norm. We note that, for instance, a set $E$ has a finite perimeter (in the sense of Caccioppoli) if and only if $\mathds 1_E$ is a function of bounded variations, so that it gives us a natural extension of the notion of perimeter to the set of admissible resources distributions.

For a general introduction to functions of bounded variations and their link with perimeter, we refer to \cite{AmbrosioFuscoPallara}.

\subsection{Main result}\label{TH}
The main result of this article is the following fragmentation property:

\begin{theorem}\label{Th:Frag}
Let, for any $\mu>0$, $m_\mu^*$ be a solution of \eqref{PV}. There holds
\begin{equation}\label{fragmentation}
\left\Vert m_\mu^*\right\Vert_{BV(\O)} \underset{\mu \to 0^+}\rightarrow +\infty.\end{equation} More precisely, we prove:
\begin{equation}\label{Eq:fragmentation2}
\forall M>0,\quad \exists \mu_M>0 \text{ s.t. }\forall \,0< \mu\leq \mu_M\, \sup_{m\in \mathcal M_M(\O)}\fint_\O \tmm<\sup_{m\in \mathcal M(\O)}\fint_\O \tmm.\end{equation}
\end{theorem}
\begin{remark}
Since $||m_\mu^*||_{L^1(\O)}=m_0$, the above statement actually says that the $TV(\O)$-seminorm of $m_\mu^*$ blows up as $\mu\to 0$.
\end{remark}

\begin{remark}\label{ReNew}
Theorem \ref{Th:Frag}, which holds for a fixed domain with a small diffusivity, can be recast in the context of a fixed diffusivity in a large domain. Indeed, considering $\O=(0;1)^n$, we can use the change of variables $\tilde x:=\frac{x}{\sqrt{\mu}}$ to state our result in the following way: considering  the logistic equation (we use $\psi_{m,\frac{1}{\sqrt{\mu}}}$ to avoid confusion with $\theta_{m,\mu}$)
\begin{equation*}
\begin{cases} \Delta \psi_{m,\frac{1}{\sqrt{\mu}}}+\psi_{m,\frac{1}{\sqrt{\mu}}}(m-\psi_{m,\frac{1}{\sqrt{\mu}}})=0\text{ in }\frac1{\sqrt{\mu}}\O\,, 
\\\frac{\partial \psi_{m,\frac1{\sqrt\mu}}}{\partial \nu}=0\text{ on }\partial \frac1{\sqrt{\mu}}\O\,, 
\\ \psi_{m,\frac{1}{\sqrt{\mu}}}>0,\end{cases}
\end{equation*}
 the optimization problem
\begin{equation}\tag{$PV'$}\label{PV'}\max_{m\in \mathcal M(\frac1{\sqrt{\mu}}\O)}\fint_{\frac1{\sqrt{\mu}}\O} \psi_{m,\frac{1}{\sqrt{\mu}}}\,, 
\end{equation}
and defining $m_{\frac{1}{\sqrt{\mu}}}$ as a maximizer of \eqref{PV'} we have, for the $TV$-seminorm,
$$\mu^{\frac{n-1}2} \left\Vert m_{\frac{1}{\sqrt{\mu}}}\right\Vert_{TV(\frac1{\sqrt{\mu}}\O)}\underset{\mu \to 0}\rightarrow +\infty.$$
\end{remark}
\color{black}

%For the sake of convenience, we introduce the notation 
%$$F(m,\mu):=\fint_\O \tmm.$$

\begin{remark} Two remarks are in order:
\begin{enumerate}
\item[$\bullet$] We could actually prove, using our method, that 
$$\overline{\underset{\mu \to 0^+}\lim}\left(\sup_{m\in \mathcal M(\O)}F(m,\mu)\right)=||F||_{L^\infty(\mathcal M(\O)\times \R_+)},$$
as will be explained later, see Remark \ref{Rem}. Proving this actually gives a (weaker) fragmentation result (i.e. one could find a subsequence of maximisers such that the corresponding sequences of $BV(\O)$-norms diverges to $+\infty$). This seems to indicate that finding the limit problem is very challenging. Finally, we were only able to prove that 
$${\underset{\mu \to 0^+}{\underline\lim}}\left(\sup_{m\in \mathcal M(\O)}F(m,\mu)\right)> m_0=\inf_{m\in \mathcal M(\O)\,, \mu\in \R_+} F(m,\mu).$$
\item[$\bullet$] Theorem \ref{Th:Frag} can be recast in terms of perimeters. In this case, one may consider the set of admissible subsets 
$$\mathcal O(\O):=\left\{E\subset \O\,, |E|=\frac{m_0}\kappa\right\}$$ and the auxiliary subsets
$$\mathcal O_M(E):=\left\{E\in \mathcal O(\O)\,, Per(E)\leq M\right\}.$$ Here, the perimeter is to be understood in the sense of Caccioppoli.

Note that, as already pointed out, the existence of a solution to 
$$\sup_{E\in \mathcal O(\O)}F(\mathds 1_E,\mu)$$ is not known for general $\mu$, see \cite{MNP,NagaharaYanagida}. However, we can prove, in the same fashion that, for every $M>0$, there exists $\mu_M>0$ such that, for any $0<\mu\leq \mu_M$, there holds
$$\sup_{E\in \mathcal O_M(\O)}F(\mathds 1_E,\mu)<\sup_{E\in \mathcal O(\O)}F(\mathds 1_E,\mu).$$

\end{enumerate}
\end{remark}

\begin{remark}\label{Re:NoPeriod}
Following \cite{LouNagaharaYanagida} in which, as mentioned, the periodic geometry of optimal resources distributions for a discretized version of \eqref{LDE} is established for certain classes of parameters, a natural question is to know whether, for the continuous problem considered here such geometric properties hold (in the one-dimensional case). Our numerical simulations in Section \ref{Num} seem to indicate that numerical maximizers are not necessarily periodic, which is in line with the simulations of \cite{MazariNadinPrivat}. This question is, to the best of our knowledge, completely open and seems highly challenging. We present a related conjecture in the conclusion. 

\end{remark}
\color{black}

\section{Proof of Theorem \ref{Th:Frag}}\label{Proof}
\subsection{The influence of periodisation}
The main idea is to exploit the non-monotonicity of the function 
$$\mu\mapsto F(m,\mu),$$ for a fixed $m$. 

We recall (see \cite{LouInfluence}) that
\begin{enumerate}
\item Setting $F(m,0)=F(m,+\infty)=m_0$ extends $F$ to a continuous function on $[0;+\infty]$.
\item $m_0$ is a strict, global minimiser of $F(m,\cdot)$ if and only if $m\not\equiv m_0.$ If $m\equiv m_0$, then $F(m,\cdot)\equiv m_0$. 
\item $F(m,\cdot)$ may have several local maxima (see \cite{LiangLou}, where a distribution $m$ such that $F(m,\cdot)$ has at least two local maxima is constructed).
\end{enumerate}
Our method of proof consists in exploiting this non-monotonicity, as well as Neumann boundary conditions and the fact that we are working in an orthotope.

Indeed, let $k\in \N$. We can extend $m$ and $\tmm$ to $[-1;1]^n$ by reflecting them across each of the axis segments $\{x_i=0, 0\leq x_i\leq 1\}\,, i=1,\dots,n$ and, then, we can extend them to $2$-periodic (in each direction) functions on $\R^n$. It then makes sense to define, for a given $m\in \mathcal M(\O)$, the functions 
$$m_k(x):=m\left(2^k x\right)\,, \theta_k(x):=\tmm(2^k x).$$ Straightforward computations show that $(m_k, \theta_k)$ solves
\begin{equation}
\begin{cases}
\displaystyle\frac{\mu}{2^{2k}}\Delta \theta_k+\theta_k(m_k-\theta_k)=0\text{ in }\O\,, 
\\
\\\displaystyle\frac{\partial \theta_k}{\partial \nu}=0\text{ on }\partial \O.
\end{cases}
\end{equation}
Furthermore, 
\begin{align*}
\fint_{[0;1]^n} \theta_k(x)dx&=\frac1{(2^k)^n}\int_{[0;2^k]^n} \tmm(y)dy
\\&=\fint_{[0;1]^n}\tmm.
\end{align*}
As a consequence of these identities, we have 
\begin{equation}\label{Eq:Period}F\left(m_k,\frac{\mu}{2^{2k}}\right)=F(m,\mu).\end{equation}
Visually, if we represent, for instance, $F(m,\cdot)$ as 
	\begin{center}%\emph{$\mu\mapsto F(m,\mu)$ is not monotone:}
		\begin{tikzpicture}
		\draw [black] plot [smooth] coordinates {(0,1.14) (0.5,1.5) (1,1.3) (2,2.5) (3,1.5) (4,1.2) (5,1.16)};
		\draw[->](0,0.5)--(0,3);
				\draw[->](0,0.5)--(5,0.5);
				%\only<3->{\draw[purple,->](0.5,0.5)--(0.5,1.5);\node[purple] at (0.5,0.3){$\mu^*$};}
				\draw[gray,dashed](0,1.14)--(5,1.14);
				 \node[gray] at (5,1) {$m_0$};
				  \node[black] at (5,0.3) {$\mu$};

		\end{tikzpicture}
		\end{center}

then $f_k:=F(m_k,\cdot)$ can be visualised as 
\begin{center}%\emph{$\mu\mapsto F(m,\mu)$ is not monotone:}
		\begin{tikzpicture}
		\draw [black] plot [smooth] coordinates {(0,1.14) (0.25,1.5) (0.5,1.3) (1,2.5) (1.5,1.5) (2,1.2) (2.5,1.16) (5,1.155)};
		\draw[->](0,0.5)--(0,3);
				\draw[->](0,0.5)--(5,0.5);
				%\only<3->{\draw[purple,->](0.5,0.5)--(0.5,1.5);\node[purple] at (0.5,0.3){$\mu^*$};}
				\draw[gray,dashed](0,1.14)--(5,1.14);
				 \node[gray] at (5,1) {$m_0$};
				  \node[black] at (5,0.3) {$\mu$};
		\end{tikzpicture}
		\end{center}
		
		%We first note that this shows that no uniform ( in $m$) convergence  of $\tmm$ to $m$ as $\mu \to 0^+$ can be proved. However, such a uniform result can be obtained in the class $\mathcal M_M(\O)$.
		
Using  \eqref{Eq:Period}, we are going to show that there exists $\eta>0$  and $\mu_\eta>0$ such that 
$$\inf_{0<\mu\leq \mu_\eta}\left(\sup_{ m\in \mathcal M(\O)}F(m,\mu)\right)\geq m_0+\eta.$$

Then, we will show that, for any $M>0$, for any $\e>0$, there exists $\mu_{M,\e}>0$ such that 
$$\inf_{0\leq \mu\leq \mu_{M,\e}}\sup_{ m\in \mathcal M_M(\O)}F(m,\mu)\leq m_0+\e.$$
The conclusion of Theorem 1 follows immediately from these two steps.

\subsection{Technical preliminaries}
\paragraph{Technical background}
We briefly recall some well-known facts about Equation \ref{LDE}. From the method of sub- and super-solutions ( we refer for instance to \cite{CantrellCosner}) we have 
$$\forall \mu\in (0;+\infty)\,, \forall m \in \mathcal M(\O)\,, 0\leq \tmm\leq \kappa.$$
Lou, in \cite{LouInfluence}, proves the following three results: first, 
\begin{equation}\forall \mu\in (0;+\infty)\,, \mu\fint_\O \frac{|\n \tmm|^2}{\tmm^2}=\fint_\O \tmm-m_0.\end{equation}
Then, 
\begin{equation}\label{Eq:CVM}
\forall m\in \mathcal M(\O)\,, \forall p\in [1;+\infty)\,, \Vert \tmm-m\Vert_{L^p(\O)}\underset{\mu \to 0}\rightarrow 0.
\end{equation}
Finally, he obtains the following estimate in \cite[Claim, Equation 2.4]{LouInfluence}: there exists a constant $C$  independent of $m$ and $\mu>0$ such that
\begin{equation}\label{Eq:Lou}\forall (m,m')\in \mathcal M(\O)^2\,, \forall \mu>0\,, \quad
\Vert \tmm-\theta_{m',\mu}\Vert_{L^1(\O)}\leq C \Vert m-m'||_{L^1(\O)}^{\frac13}.\end{equation}
Although Lou, in \cite{LouInfluence}, assumes that $m'$ is regular, it is readily checked that his proof does not depend on the smoothness of $m'$ and can be extended to all elements of $\mathcal M(\O)$ in a straightforward way.

\paragraph{Uniform convergence in $\mathcal M_M(\O)$ (as $\mu\to 0^+$)}
Our goal is to make the convergence result \eqref{Eq:CVM} uniform in $m\in \mathcal M_M(\O)$. This is the content of the following Lemma:

\begin{lemma}\label{Le:Unif}
For any $M>0$, the convergence result \eqref{Eq:CVM} is uniform in $\mathcal M_M(\O)$ in the following sense: let $M>0$ be fixed, then
\begin{equation}\label{Eq:Le1}
\forall \e_0>0\,, \exists \mu_{M,\e_0} \text{ s.t. } \forall m\in \mathcal M_M(\O),\quad   \forall\, 0\leq \mu\leq \mu_{M,\e_0} \quad \Vert \tmm-m\Vert_{L^1(\O)}\leq \e_0.\end{equation}
\end{lemma}

\begin{proof}[Proof of Lemma \ref{Le:Unif}]
We argue by contradiction. If we assume that \eqref{Eq:Le1} does not hold then there exists $\e_0>0$ and a sequence $\{m_k,\mu_k\}\in \left(\mathcal M_M(\O)\times (0;+\infty)\right)^\N$ such that:
\begin{enumerate}
\item $\{\mu_k\}_{k\in \N}$ is decreasing and converging to 0,
\item There holds:
\begin{equation}\label{Eq:Review}\forall k\in \N\,, \Vert\theta_{m_k,\mu_k}-m_k\Vert_{L^1(\O)}\geq \e_0>0.\end{equation}
\end{enumerate}
The embedding $BV(\O)\hookrightarrow L^1(\O)$ is compact. Hence, there exists $m_\infty\in \mathcal M_M(\O)$ such that 
\begin{equation}\label{Eq:Int}m_k\underset{k\to \infty}\rightarrow m_\infty\text{ strongly in } L^1(\O). \end{equation}
Thus, we can write, for any $k\in \N$,
\begin{align*}
\Vert \theta_{m_\infty,\mu_k}-m_\infty\Vert_{L^1(\O)}&\geq \Vert \theta_{m_k,\mu_k}-m_k\Vert_{L^1(\O)}-\Vert m_k-m_\infty\Vert_{L^1(\O)}-\Vert \theta_{m_k,\mu_k}-\theta_{m_\infty,\mu_k}\Vert_{L^1(\O)}
\\&\geq \e_0+\underset{k\to \infty}o(1)+\underset{k\to \infty}o(1)\text{ by, successively, \eqref{Eq:Review}, \eqref{Eq:Int} and \eqref{Eq:Lou}.}
\end{align*}

This is in  contradiction with  \eqref{Eq:CVM}.  Lemma \ref{Le:Unif} is proved.

\end{proof}

\paragraph{Estimating $\underset{\mu \to 0}{\underline \lim}\left(\sup_{\mathcal M(\O)} F(\cdot,\mu)\right)$}
The goal of this paragraph is the following Lemma:
\begin{lemma}\label{Le:2}
There exist $\eta>0$ and $\mu_\eta>0$ such that 
\begin{equation}\label{Eq:L2}
\inf_{0< \mu\leq \mu_\eta}\left(\sup_{ m\in \mathcal M(\O)}F(m,\mu)\right)\geq m_0+\eta.\end{equation}

\end{lemma}

\begin{proof}[Proof of Lemma \ref{Le:2}]

Let $m\in \mathcal M(\O)$ be any non-constant admissible resources distribution. We know that $m_0$ is a strict local minimum of $F(m,\cdot)$ on $[0;+\infty)$, and that $\mu \mapsto F(m,\mu)$ is continuous on $[0;+\infty)$.  Let $\underline \mu>0$ be a real number and consider the interval
$$I_0:=\left[\underline \mu,4\underline\mu\right].$$

Since $m_0$ is only reached at $\mu=0$ and $\mu=\infty$, it follows that 
$$\inf_{\mu \in I_0}F(m,\mu)>m_0.$$ Thus, let $\eta>0$ be such that
$$\inf_{\mu \in I_0}F(m,\mu)\geq m_0+\eta.$$
We then consider, for any $k\in \N$, the interval
$$I_k:=\left[\frac{\underline\mu}{2^{2k}}, \frac{4\underline\mu}{ 2^{2k}}\right].$$ 

We first remark that, once again setting $m_k(\cdot)=m(2^k\cdot)$ and, thanks to \eqref{Eq:Period}, we have 
$$F(m_k,\cdot)(I_k)=F(m,\cdot)(I_0)$$ so that 
$$\forall k \in \N\,, \inf_{\mu \in I_k}F(m_k,\mu)=\inf_{\mu \in I_0}F(m,\cdot)\geq m_0+\eta.$$
Hence, 
\begin{equation}\label{eq3}\inf_{\mu\in I_k}\left(\sup_{m\in \mathcal M(\O)}F(m,\mu)\right)\geq m_0+\eta.\end{equation}

Now, we have built our sequence in such a way that 
$$\sup(I_{k+1})=\inf(I_k).$$
Hence, setting 
$$I_\infty:=\displaystyle\bigcup_{k=1}^\infty I_k,$$
we can write 
$$I_\infty=\left(0,4\underline \mu\right) $$ and, as a consequence of \eqref{eq3}, 
$$\inf_{\mu \in I_\infty}\left(\sup_{m\in \mathcal M(\O)}F(m,\mu)\right)\geq m_0+\eta.$$

This concludes the proof.

\end{proof}
\begin{remark}\label{Rem}
We can use the same method to prove that $$\overline{\underset{\mu \to 0^+}\lim}\left(\sup_{m\in \mathcal M(\O)}F(m,\mu)\right)=||F||_{L^\infty(\mathcal M(\O)\times \R_+)}.$$
Indeed, consider the problem
$$\sup_{\mu>0,m\in \mathcal M(\O)}F(m,\mu)=F(m^*,\mu^*).$$ Uniqueness does not hold for this problem, because of the periodisation process we used above.

One can actually see that we can choose $\mu^*>0$. In this case, considering the sequence $\left(m_k^*,\frac{\mu^*}{2^{2k}}\right)_{k\in \N}$ immediately gives the result.

\end{remark}
\subsection{The proof}

\begin{proof}[Proof of Theorem \ref{Th:Frag}]
Let $M>0$ be fixed. We are going to prove that there exists $\mu_M>0$ such that, for any $0<\mu\leq \mu_M$, 
\begin{equation}\label{Eq:goal}
\sup_{m\in \mathcal M_M(\O)}F(m,\mu)<\sup_{m\in \mathcal M(\O)}F(m,\mu).\end{equation}
Theorem \ref{Th:Frag} follows immediately from \eqref{Eq:goal}.

Let $\eta>0$ and $\mu_\eta>0$   (given by Lemma \ref{Le:2}) be fixed throughout the rest of this demonstration:
\begin{equation}\label{Eq:4}
\forall 0<\mu\leq \mu_\eta\,, \sup_{m\in \mathcal M(\O)}F(m,\mu)\geq m_0+\eta.
\end{equation}

From Lemma \ref{Le:Unif}, there exists $\mu_{M,\frac\eta2}>0$ such that, for any $0<\mu\leq \mu_{M,\frac\eta2}$, we have 
$$\sup_{m\in \mathcal M_M(\O)}\Vert \tmm-m\Vert_{L^1(\O)}\leq \frac\eta2.$$
Thus
$$\forall \mu \leq \mu_{M,\frac\eta2}\,, \forall m\in \mathcal M_M(\O)\,, \fint_\O \tmm\leq m_0+\frac\eta2.$$
Plugging this in \eqref{Eq:4} then proves that, for $\mu\leq \min(\mu_\eta,\mu_{M,\frac\eta2})$, no solution $m_\mu^*$ of \eqref{PV} can belong to $\mathcal M_M(\O)$. This concludes the proof.

\end{proof}

\begin{remark}
We quickly comment on the following, expected, remark: not only does the $BV(\O)$-norm blow up, but also, every  $X(\O)$-norm, where $X(\O)$ is compactly embedded in $L^1(\O)$. Indeed, the only part where $BV$ is used is in the proof of Lemma \ref{Le:Unif}, and it is used to get strong $L^1(\O)$ convergence.
\end{remark}

\section{Numerical simulations}\label{Num}\textcolor{black}{
We present several numerical simulations in order to emphasise the results of Theorem \ref{Th:Frag}. All of these simulations were obtained using Ipopt \cite{IPOPT}. These simulations confirm our result. In the one-dimensional case, we even have a stronger information: if we define $m_\mu^*$ as the optimal resources distribution and if we assume (which seems to be validated by numerical simulations) that $m_\mu^*=\mathds 1_{E_\mu^*}$ is  a characteristic function, then Theorem \ref{Th:Frag}, which states that $Per(E_\mu^*)+|E_\mu^*|=\Vert m_\mu^*\Vert_{BV}$ goes to $+\infty$ as $\mu \to 0$, implies that the number of connected components of $E_\mu^*$ goes to $+\infty$. We however note that in the two-dimensional case, such an explosion of the number of connected components is not implied by Theorem \ref{Th:Frag}. It should also be noted that the results of our one-dimensional simulations do not necessarily exhibit periodic structures.}

Let $h>0$ be the discretization parameter. We work with a uniform space discretisation of size $h$.  Since, numerically, such optimisation problems can be very complicated, we run our optimisation program with different initial guesses $\boldsymbol{m}_h^{(1)}, \dots, \boldsymbol{m}_h^{(k)},\dots$ to obtain, for each initial guess, a potential candidate to be the optimiser. We then select, among these candidates, the optimal one by comparing the value of the criteria and, to check our results, we apply a gradient descent as a final step.

The simulations are done in the following way (we only present it in the one-dimensional case):
\begin{itemize}
 \item \textit{ Generating  random initial guesses $(\boldsymbol{\theta}_h^{(k)},\boldsymbol{m}_h^{(k)})$ }. We generate a random sample of initial  guesses $\boldsymbol{m}_h^{(k)}$ by randomising their first five Fourier coefficients on each discretisation interval. In other words, we define
 $$\boldsymbol{m}_h^{(k)}=m_{h,i}^{(k)}\text{ on }I_i:=\left[hi;h(i+1)\right]$$ where each of the $m_{h,i}^{(k)}$ is a random function generated as follows in the one-dimensional case: \begin{equation}
  m_{h,i}^{(k)}=a_0+\sum_{j=1}^5 a_j\sin(j\pi ih)+b_j\cos(j\pi ih)
 \end{equation}
 where $a_j$ and $b_j$ are uniform random variables with values in $[-0.5,0.5]$. To ensure that the resulting function $\boldsymbol{m}_h^{(k)}$ satisfies the constraint $\boldsymbol{m}_h^{(k)}\in \mathcal M(0;1)$, we apply an affine transformation 
 \begin{equation*}
  \begin{cases}
   T\left(\boldsymbol{m}_h^{(k)}\right)=a\boldsymbol{m}_h^{(k)}+b\\
   a=\min\left(\left|\frac{\kappa-m_0}{\max_i(m^{(k)}_{h,i}-\sum_i m^{(k)}_{h,i})}\right|,\left|\frac{-m_0}{\min_i(m^{(k)}_{h,i}-\sum_i m^{(k)}_{h,i})}\right|\right)\\
   b=m_0-a\displaystyle\sum_i m^{(k)}_{h,i}.
  \end{cases}
 \end{equation*}
 
 The resulting function satisfies $\boldsymbol{m}_h^{(k)}\in \mathcal{M}(\Omega)$.
 
 Now, to each of these random initial guess we need to associate an initial guess for the solution of the partial differential equation. We choose an energetic approach: we minimise with Ipopt the discretised energy functional  associated with Equation \eqref{LDE} to obtain $\boldsymbol{\theta}^{(k)}_h$, which is a piecewise constant function: $\boldsymbol{\theta}^{(k)}_h=\theta_{h,i}^{(k)}$ on $[ih;(i+1)h]$; in other words, $\boldsymbol{\theta}^{(k)}_h$ is the minimiser of 
 \begin{equation}\label{numfun}
   J^h(\boldsymbol{\theta}_h)=\frac{\mu}{2}\boldsymbol{\theta}_h\left(-\triangle_h\right)\boldsymbol{\theta}_h-\sum_{i=1}^N \left(\frac{1}{2}\theta_{h,i}^{2}m^{(k)}_{h,i}-\frac{1}{3}\theta_{h,i}^3\right).
 \end{equation}
 where $\triangle_h$ is the discrete Laplacian with Neumann boundary conditions, in 1D:
  \begin{equation}
  \triangle_h:=\frac{1}{h^2}\begin{pmatrix}
                -2 & 2 &0 & 0 & 0 & &\cdots & &0\\
                1 & -2 &1 & 0 & 0 & &\cdots & &0\\
                0 & 1 &-2 & 1 & 0 & &\cdots & &0\\
                  &  & &  &  & & & &\\
                 &  & & \ddots & \ddots &\ddots & & &\\
                 &  & &  &  & & & &\\
                0 & & & \cdots &  & 0& 1& -2&1\\
                0 & & &  \cdots&  & & 0& 2&-2
               \end{pmatrix}
 \end{equation}
 
 In then end, we get an initial random guess for an optimiser, which we denote $(\boldsymbol{\theta}_h,\boldsymbol{m}_h)$.

 \item \textit{Optimisation under a finite difference scheme constraint}  We use Ipopt to maximise the total population $\sum \theta^{(k)}_{h,i}$ for every $k$ with respect to $\boldsymbol{m}_h$. We implement the partial differential equation \eqref{LDE} as a constraint in the scheme:$$\mu(-\triangle_h\boldsymbol{\theta}_{h})_i=\theta_{h,i} m_{h,i}-\theta_{h,i}^2,$$ and, obviously, the constraint  $\boldsymbol{m}\in \mathcal M(0;1)$. Among all random initialisations, we choose the best solution.
 \item \textit{Gradient descent} We recall that, in this context, the adjoint state for the variational problem \eqref{PV} is the function $p \theta_{m,\mu}$ where $p$ solves 
 $$-\mu \Delta p-p(m-2\theta_{m,\mu})=1$$ with Neumann boundary conditions; in other words, for an admissible perturbation $\xi$ at an admissible resources distribution $m$ (i.e. for every $t$ small enough, $m+t\xi\in \mathcal M(0;1)$) the derivative of the criterion at $m$ in the direction $\xi$ is 
 $$\int_\O p\theta_{m,\mu}\xi.$$ We refer to \cite{Ding2010,MNP,NagaharaYanagida}.\color{black}
 
 First compute, with the same space discretisation, the discretised adjoint state $\boldsymbol{p}_h$:
\begin{equation}
 -\mu\triangle_h \boldsymbol{p}_h-diag(\boldsymbol{p}_h)(\boldsymbol{m}_h-2\boldsymbol{\theta}_h)=\boldsymbol{1}
\end{equation}
and we find the admissible perturbation $\boldsymbol{\xi}_h$ that gives the highest rise of the total population via maximizing with Ipopt the following quantity:
 \begin{equation}
  \max_{0\leq m_{h,i}+\xi_{h,i}\leq \kappa, \sum_i \xi_{h,i}=0}\sum_i p_{h,i}\theta_{h,i} \xi_{h,i},
 \end{equation}
 which corresponds to the highest directional derivative with respect to $\boldsymbol{m}$, and we apply the gradient  descent with Armijo rule. We use a classical stopping criterion, and display the results.
\end{itemize}

\subsection{Simulations in the one-dimensional case}

For one-dimensional simulations, we work in 
$$\O=(0;1),$$
with $\kappa=1$ and $N_x=1000$ discretization points
. For each value of the parameter $\mu$, we represent, on the same picture the optimal resources distribution $m_\mu^*$ (the blue zones correspond to $m=1$), which we observe, in each of our case, to be a bang-bang function, and  the corresponding solution $\theta_{m_\mu^*,\mu}$ of \eqref{LDE}.

In order to emphasise the influence of the parameter $m_0$ on the qualitative properties of optimal resources distributions, we present two different values of $m_0$.
 \newpage
\subsubsection{$\kappa=1$, $m_0=0.3$} 

\begin{figure}[h!]
\begin{center}
\includegraphics[width=6cm]{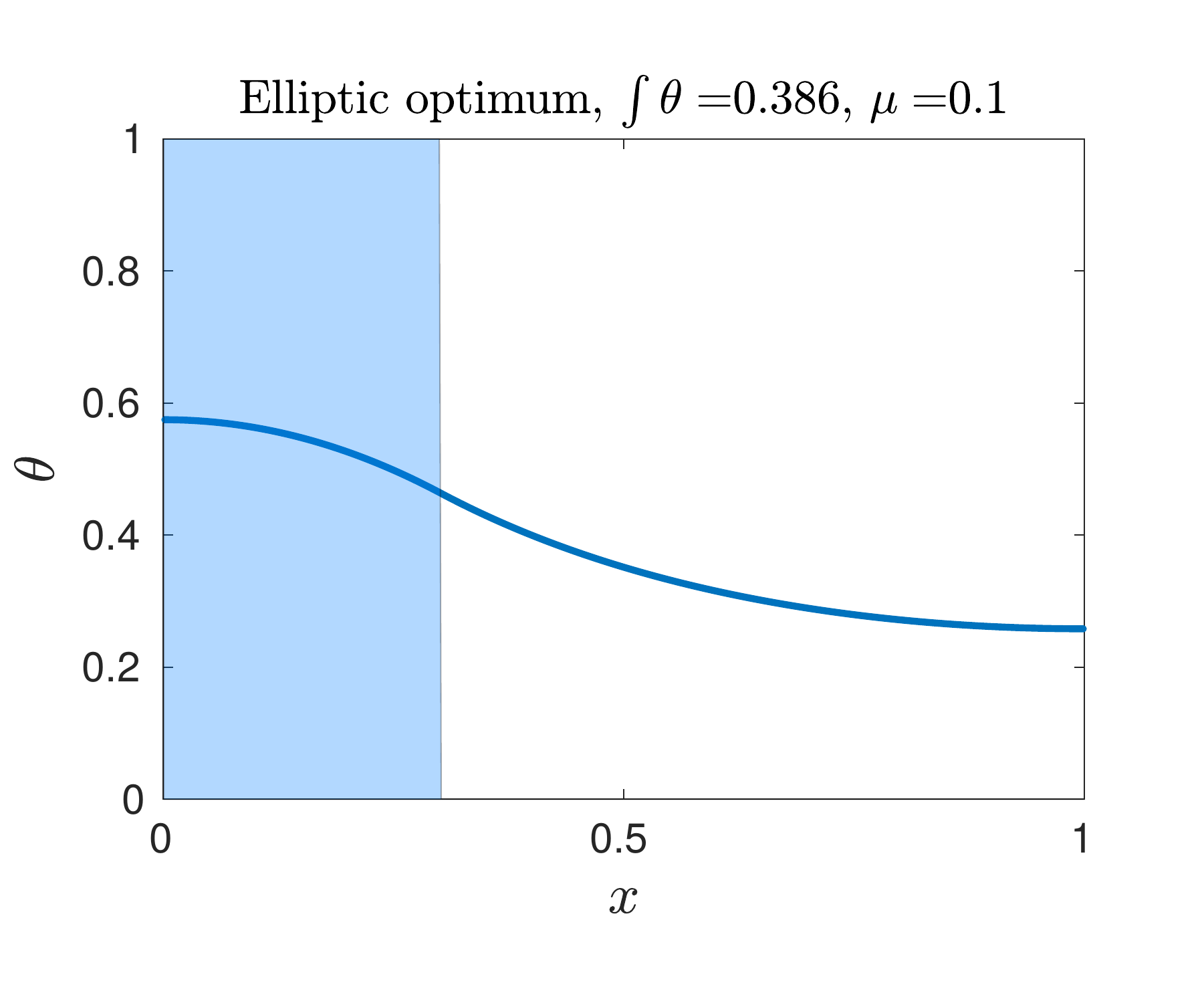}
\includegraphics[width=6cm]{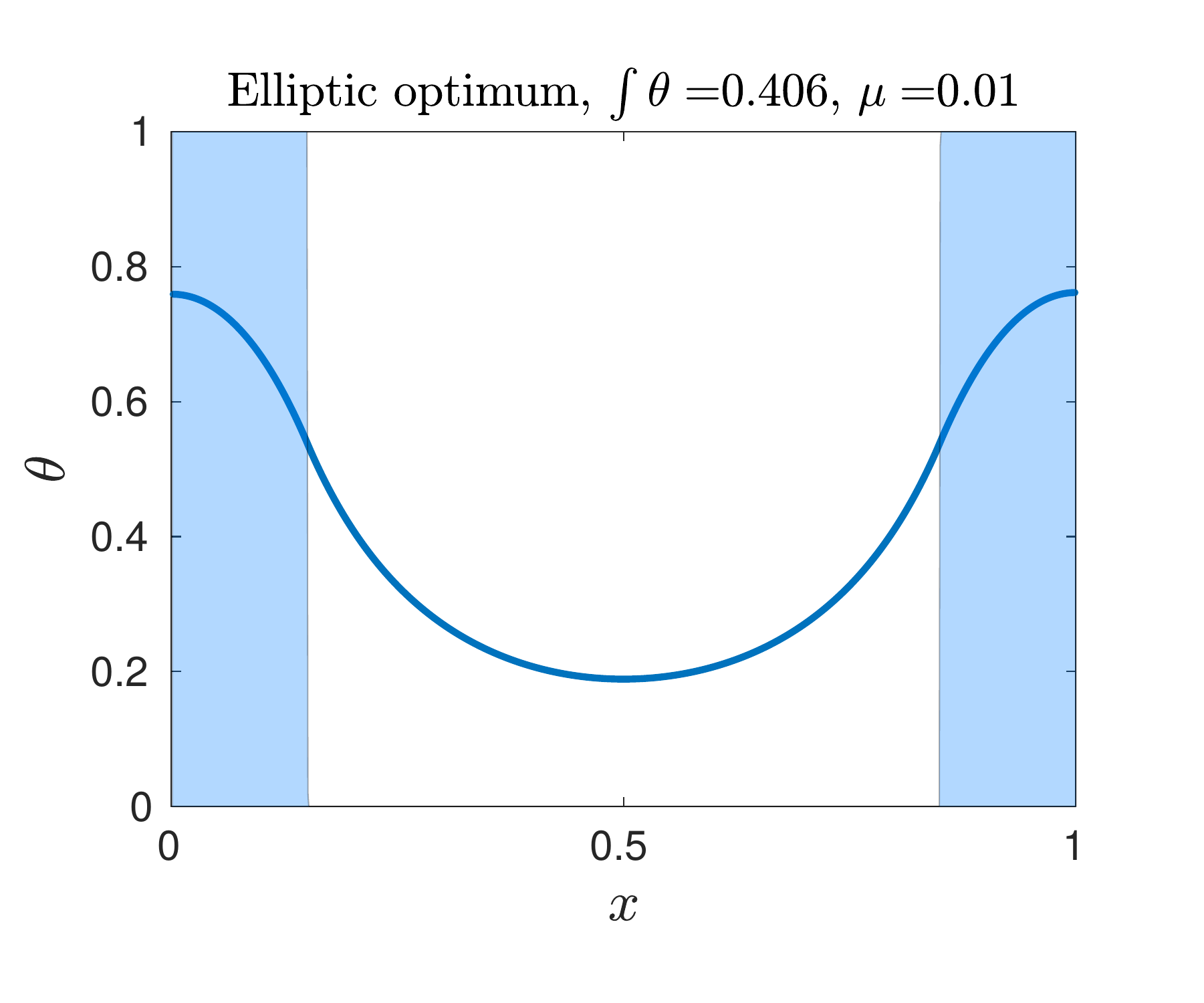}
\caption{}%{For a 'large' diffusivity, we obtain a very concentrated optimal resources distribution.}
\end{center}
\end{figure}

\begin{figure}[h!]
\begin{center}
\includegraphics[width=6cm]{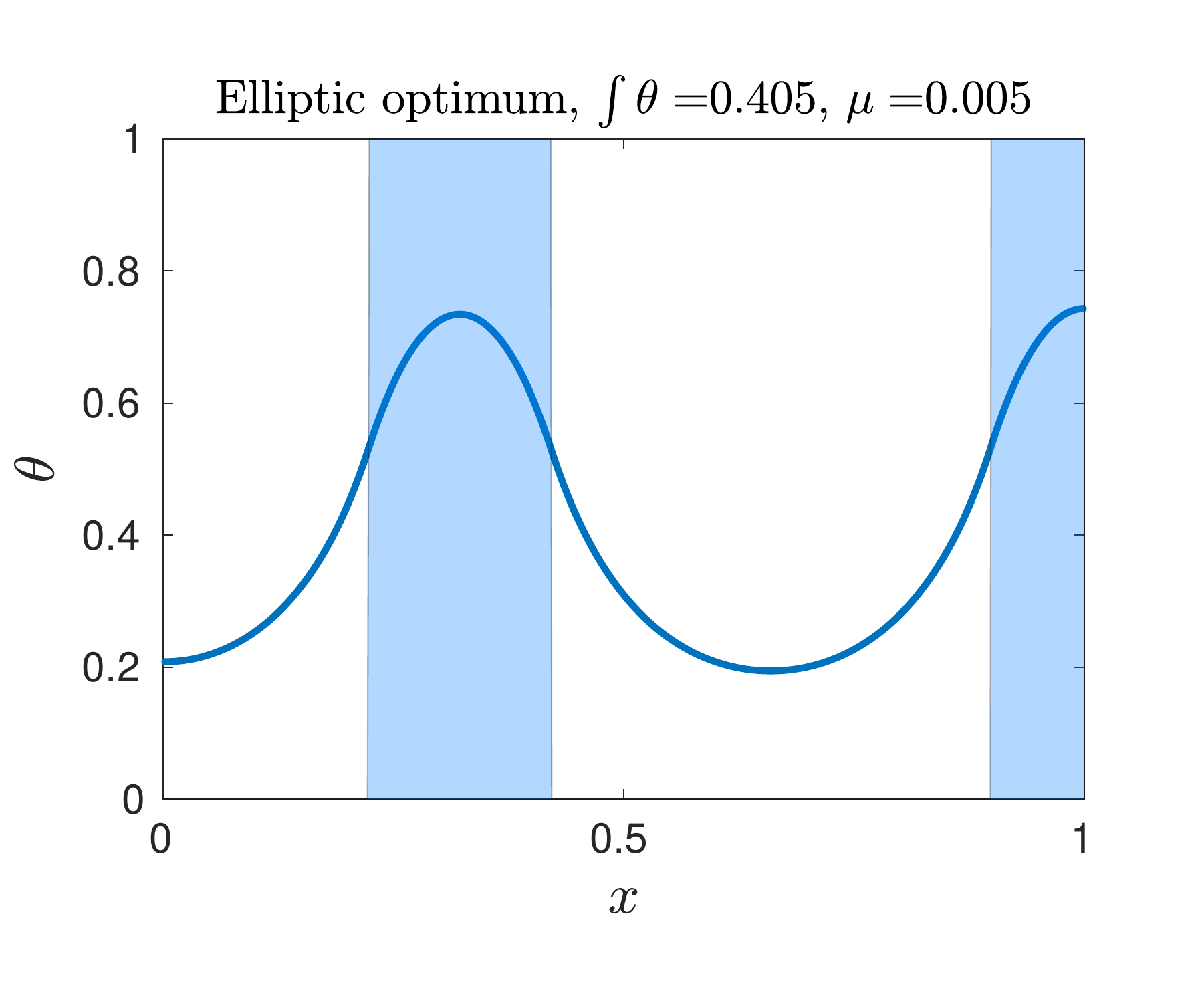}
\includegraphics[width=6cm]{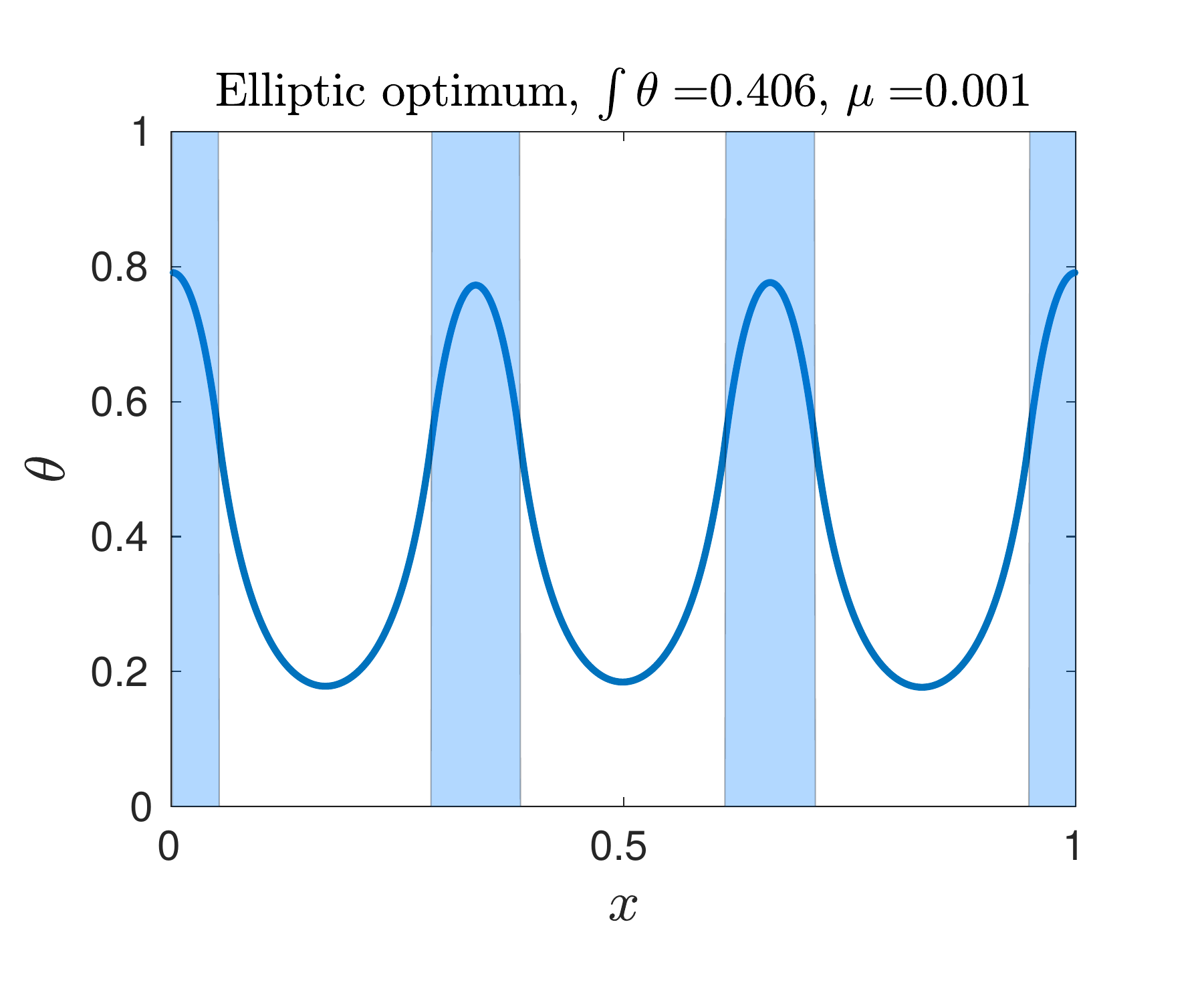}
\caption{}%{For a 'large' diffusivity, we obtain a very concentrated optimal resources distribution.}
\end{center}
\end{figure}

\begin{figure}[h!]
\begin{center}
\includegraphics[width=6cm]{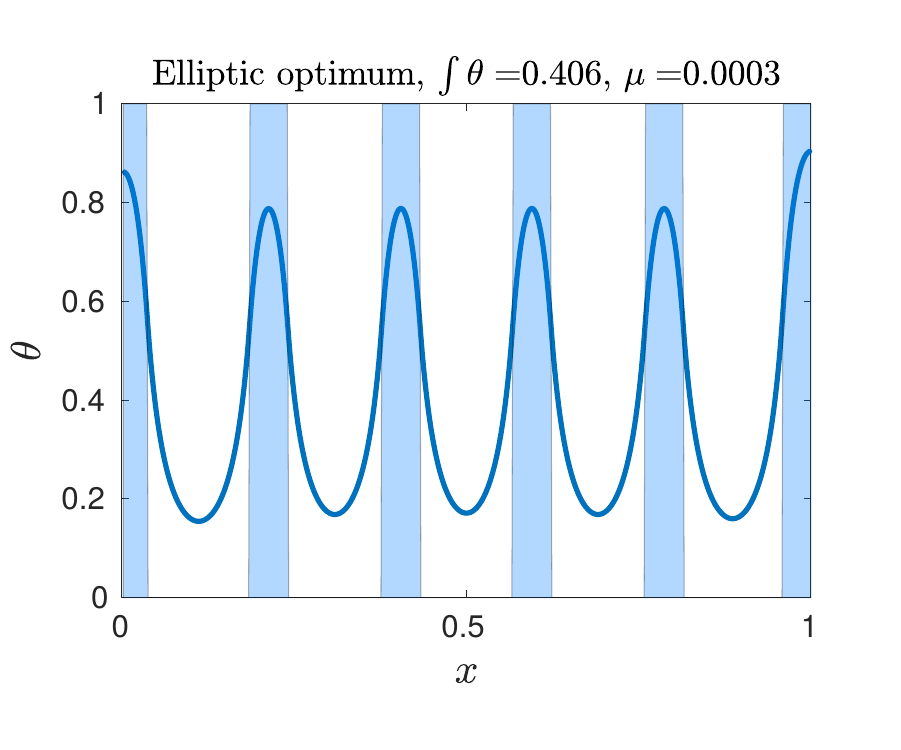}
\includegraphics[width=6cm]{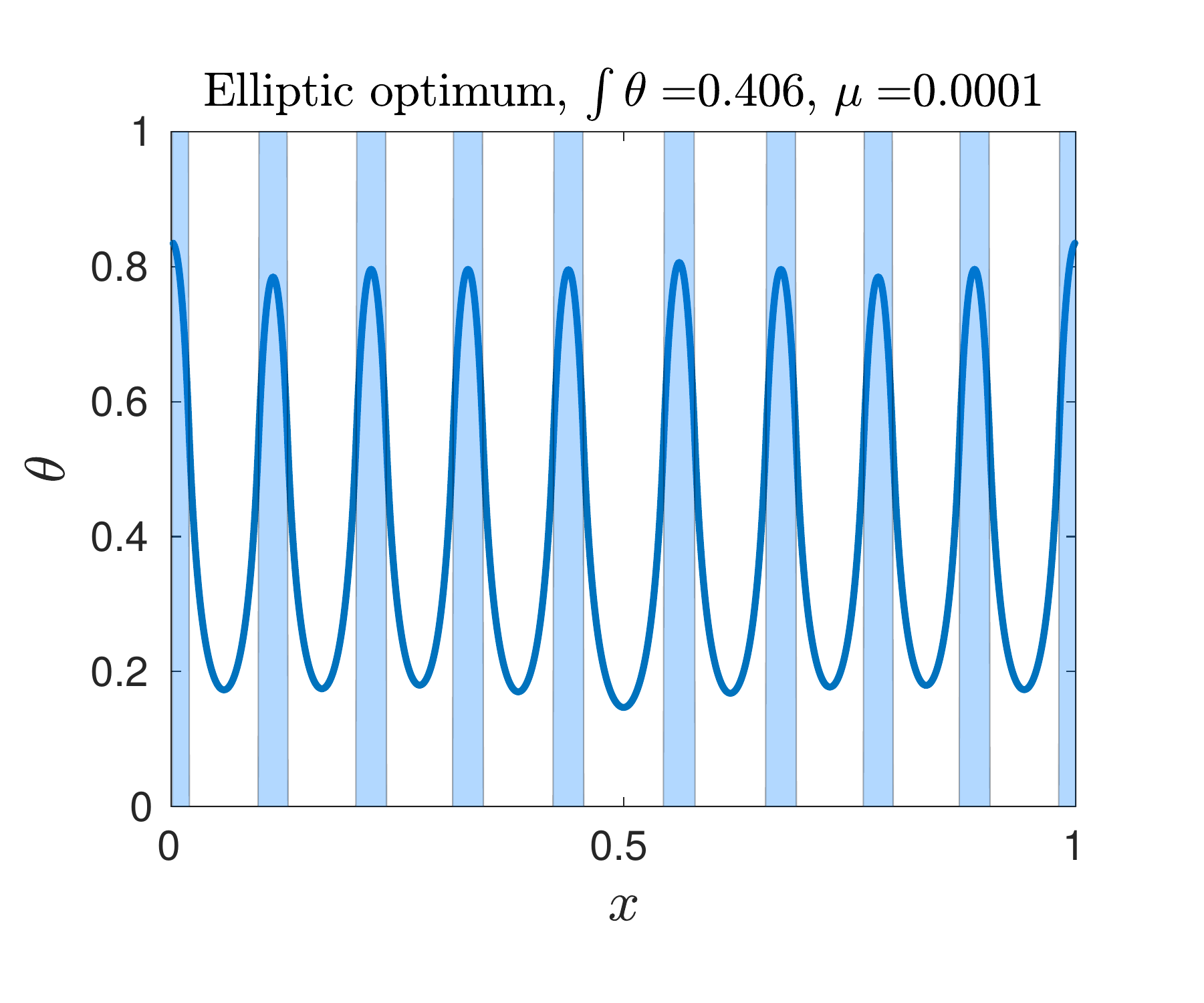}
\caption{}%{For a 'large' diffusivity, we obtain a very concentrated optimal resources distribution.}
\end{center}
\end{figure}
\newpage

\subsubsection{$\kappa=1$, $m_0=0.6$} 

\begin{figure}[h!]
\begin{center}
\includegraphics[width=6cm]{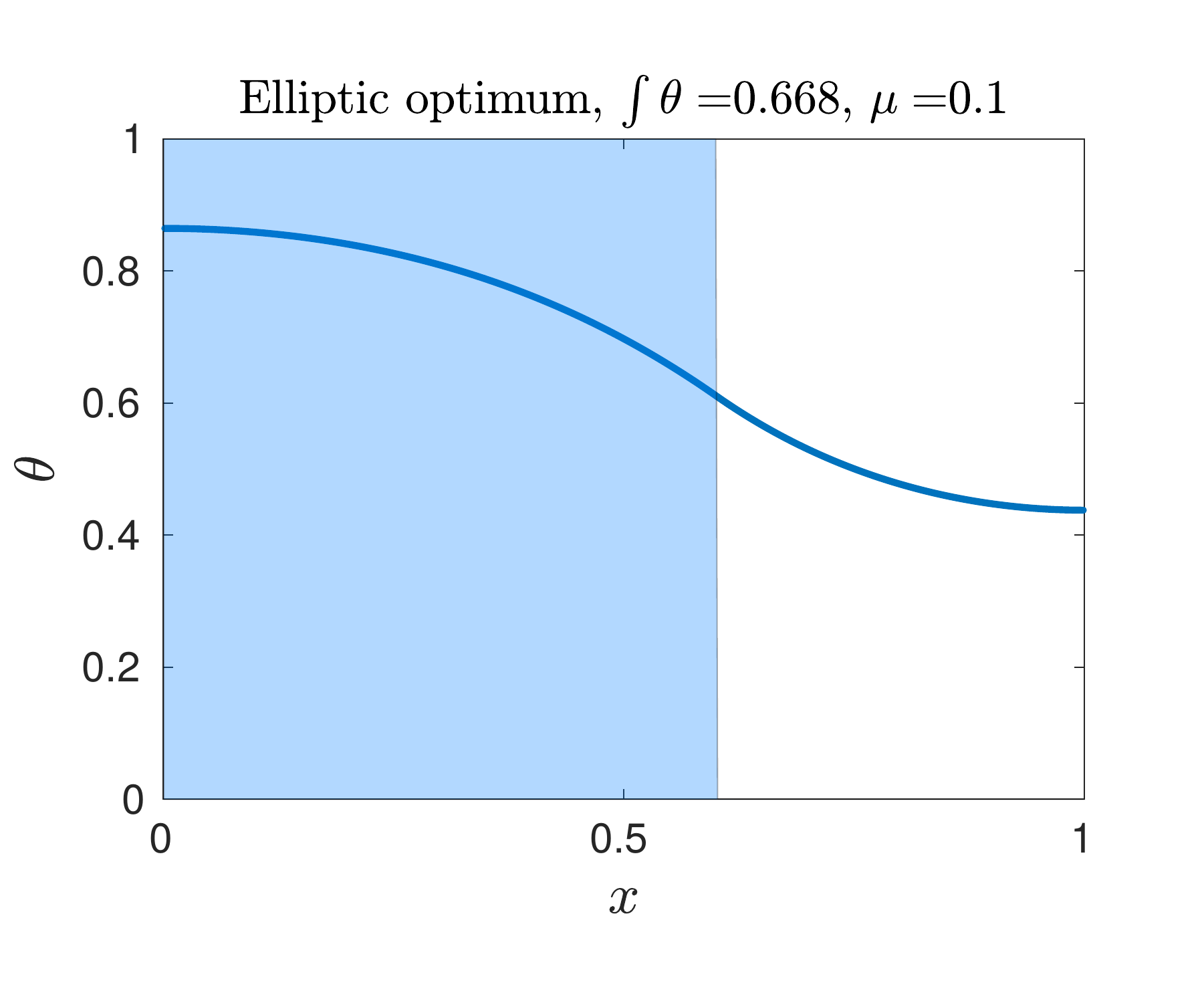}
\includegraphics[width=6cm]{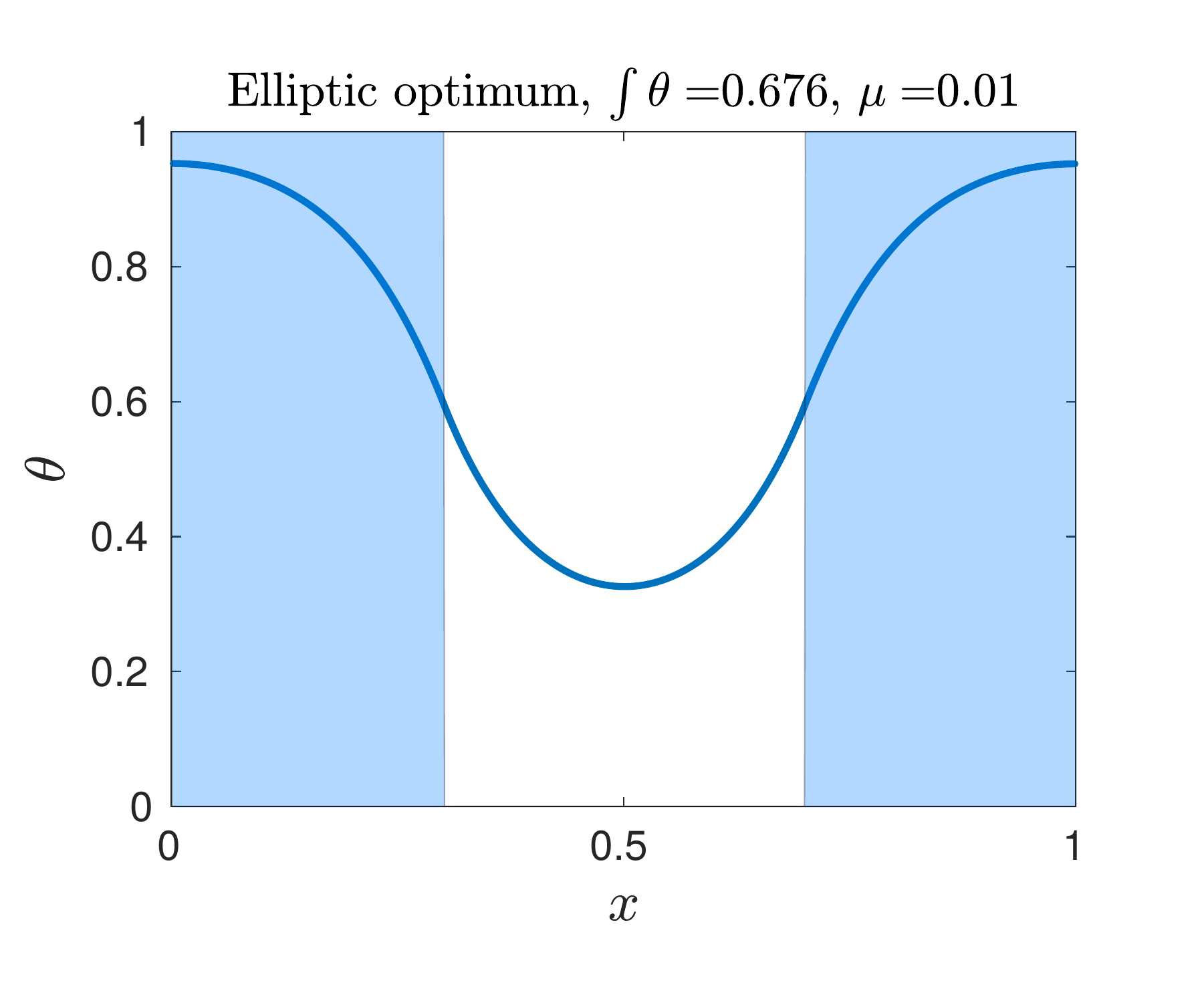}
\caption{}%{For a 'large' diffusivity, we obtain a very concentrated optimal resources distribution.}
\end{center}
\end{figure}

\begin{figure}[h!]
\begin{center}
\includegraphics[width=6cm]{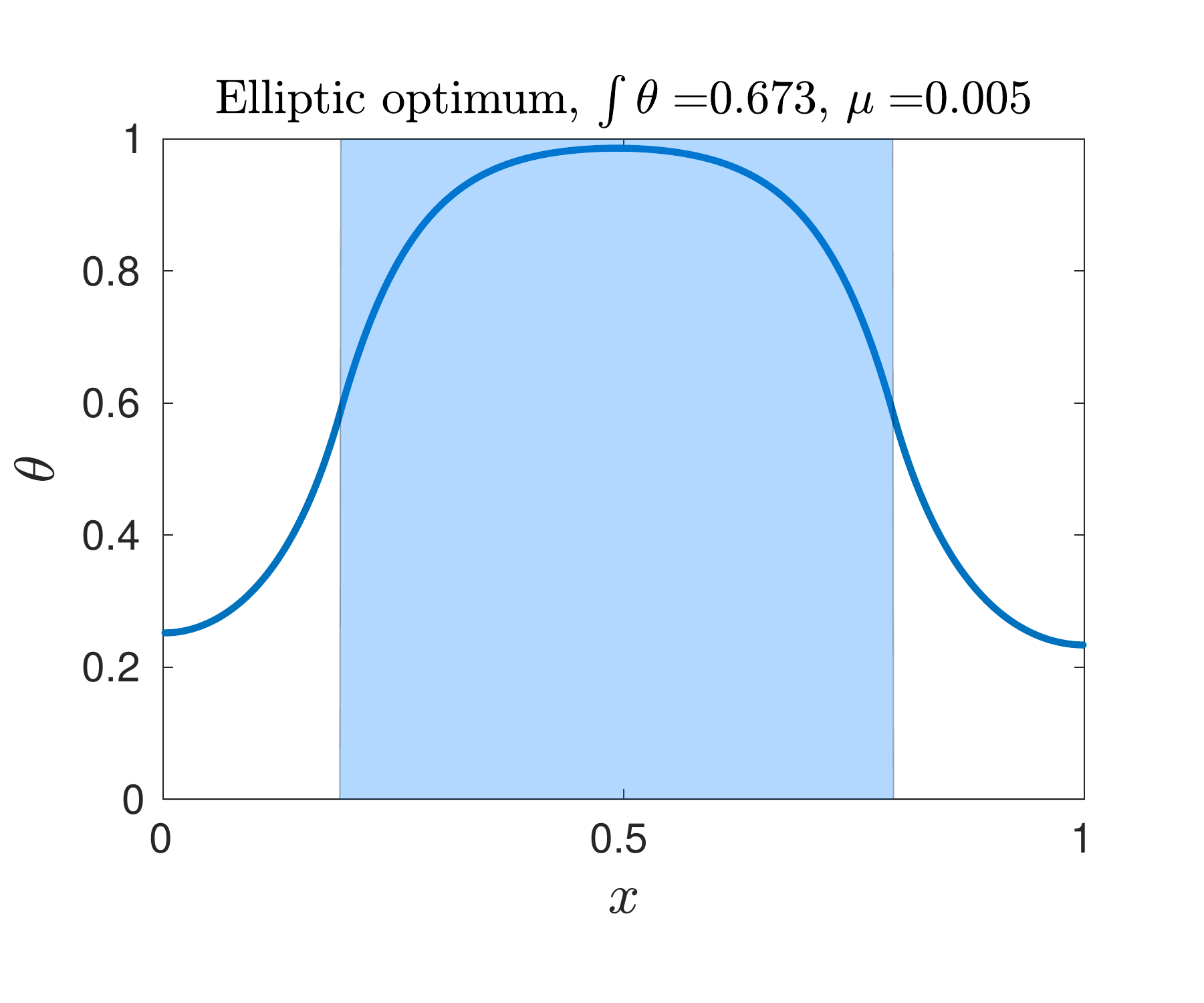}
\includegraphics[width=6cm]{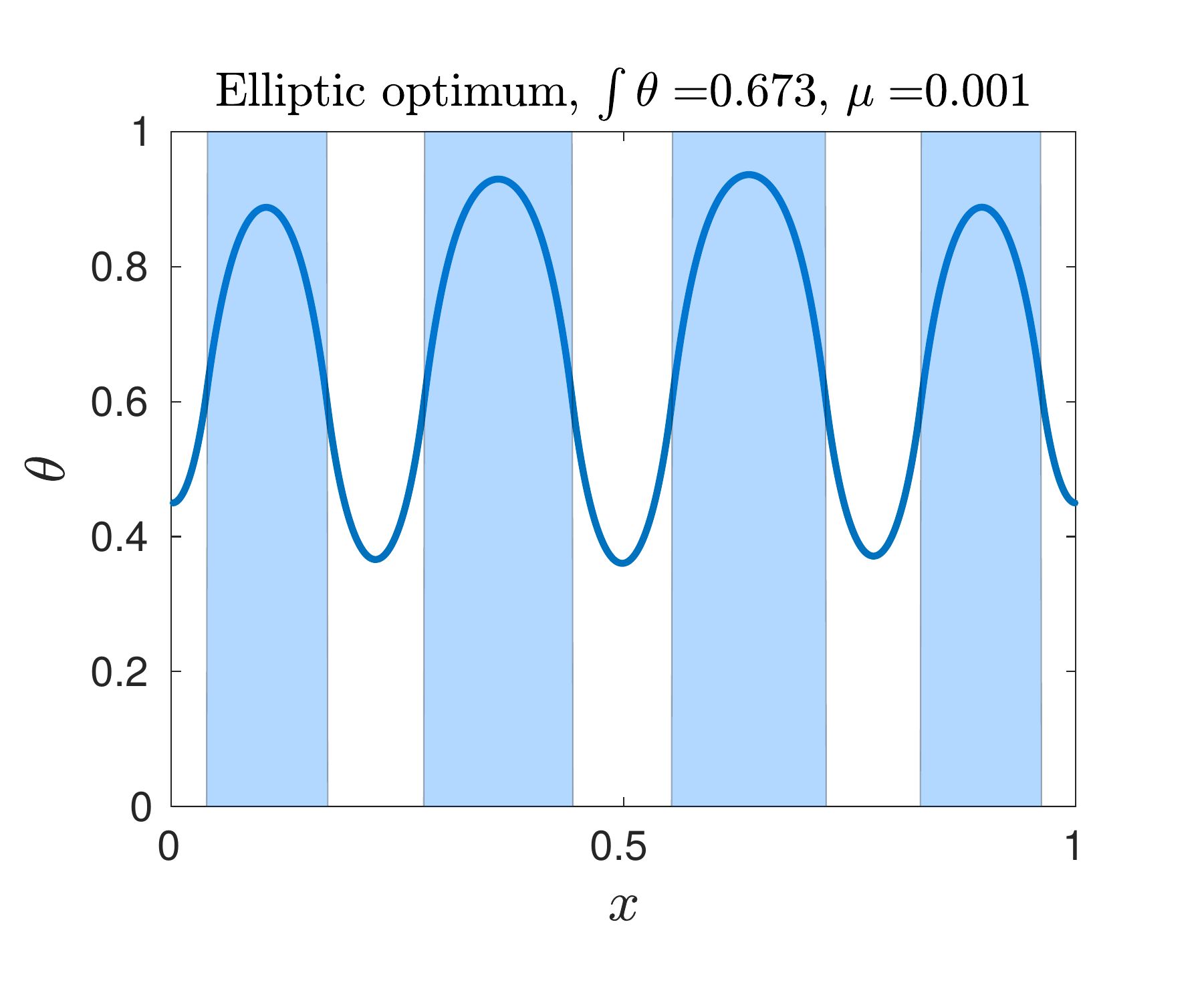}
\caption{}%{For a 'large' diffusivity, we obtain a very concentrated optimal resources distribution.}
\end{center}
\end{figure}

\begin{figure}[h!]
\begin{center}
\includegraphics[width=6cm]{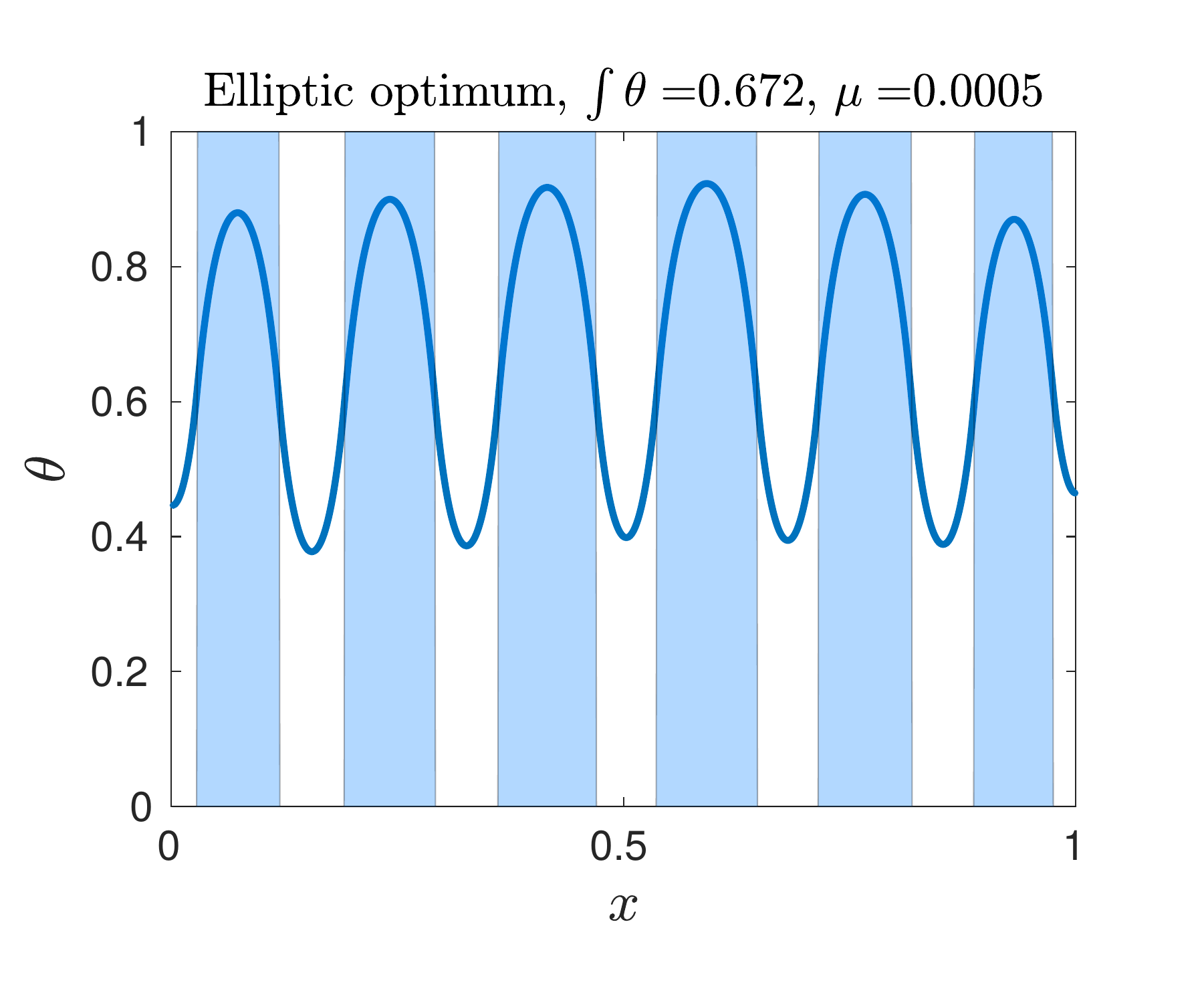}
\includegraphics[width=6cm]{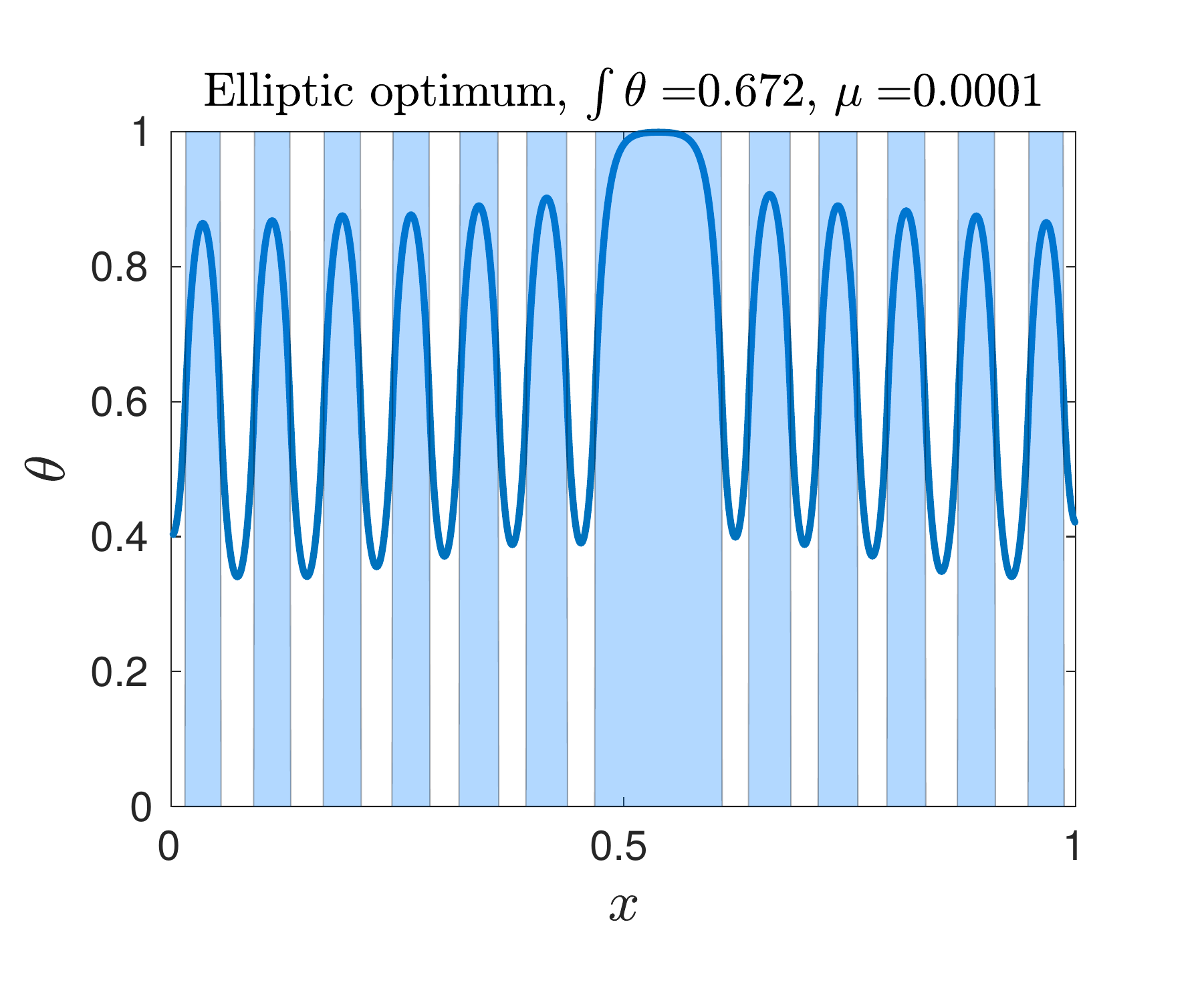}
\caption{}%{For a 'large' diffusivity, we obtain a very concentrated optimal resources distribution.}
\end{center}
\end{figure}

\newpage\subsection{Simulations in the two-dimensional case}

For two-dimensional simulations, we work in 
$$\O=[0;1]^2,$$
with $\kappa=1$. For each value of the parameter $\mu$, we represent, on the left picture, the optimal resources distribution $m_\mu^*$, which we observe, in each of our case, to be a bang-bang function. On the right, we represent the corresponding solution $\theta_{m_\mu^*,\mu}$ of \eqref{LDE}.

In order to emphasise the influence of the parameter $m_0$ on the qualitative properties of optimal resources distributions, we present, as in the one-dimensional case, two different cases. We  once again highlight the fact that these simulations prohibit, at a theoretical level, the use of rearrangements to derive qualitative properties but we do notice, in this two dimensional case, the presence of many symmetries. It is a very challenging and interesting  project to obtain symmetry properties for this kind of problems. The number of discretisation points in the $x$ and $y$ variable are $N_x=N_y=60$; the method is otherwise similar to that in the one-dimensional case.

\subsubsection{$\kappa=1$, $m_0=0.3$} 

\begin{figure}[h!]
\begin{center}
\includegraphics[width=6cm]{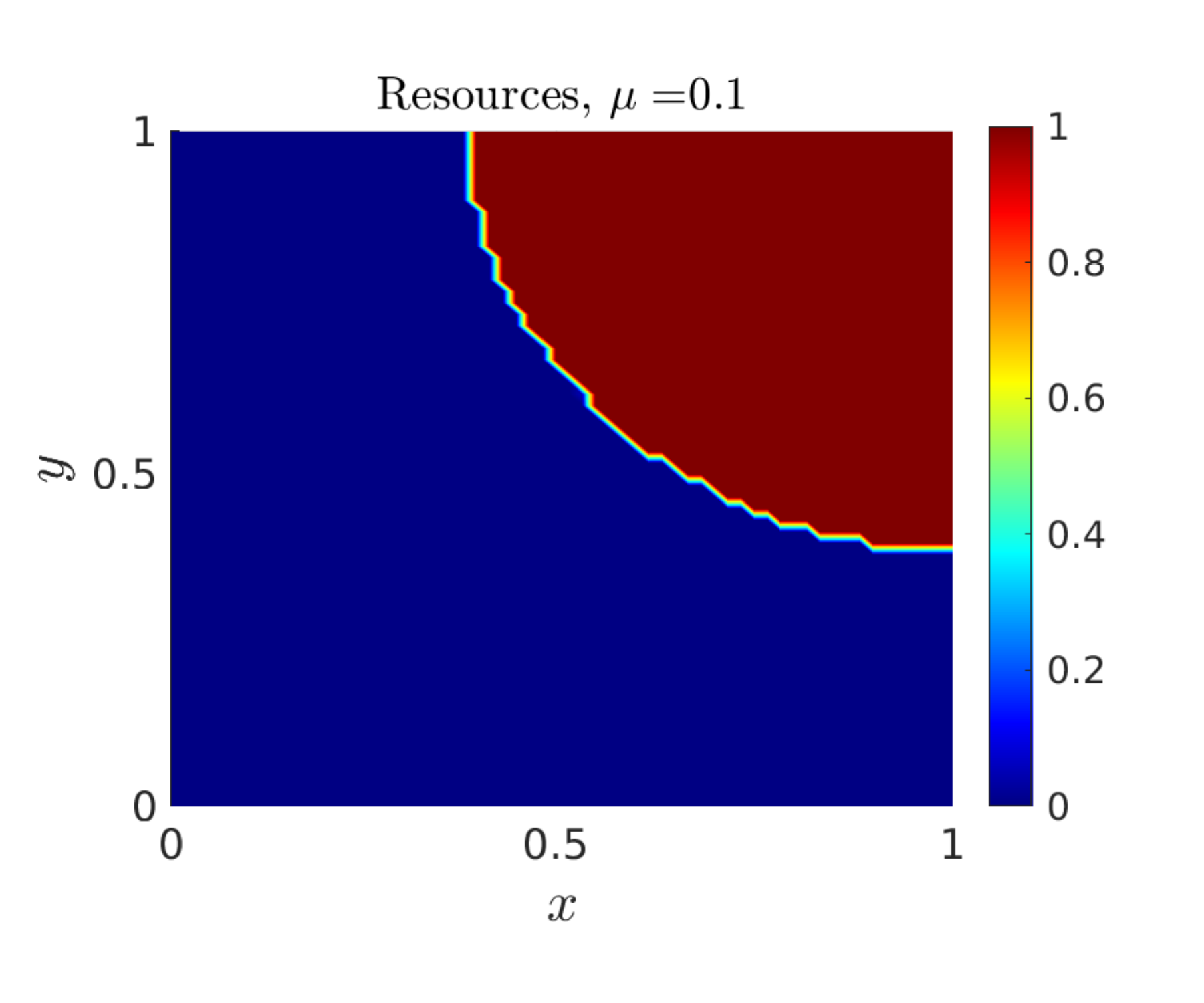}\hspace{1cm}
\includegraphics[width=6cm]{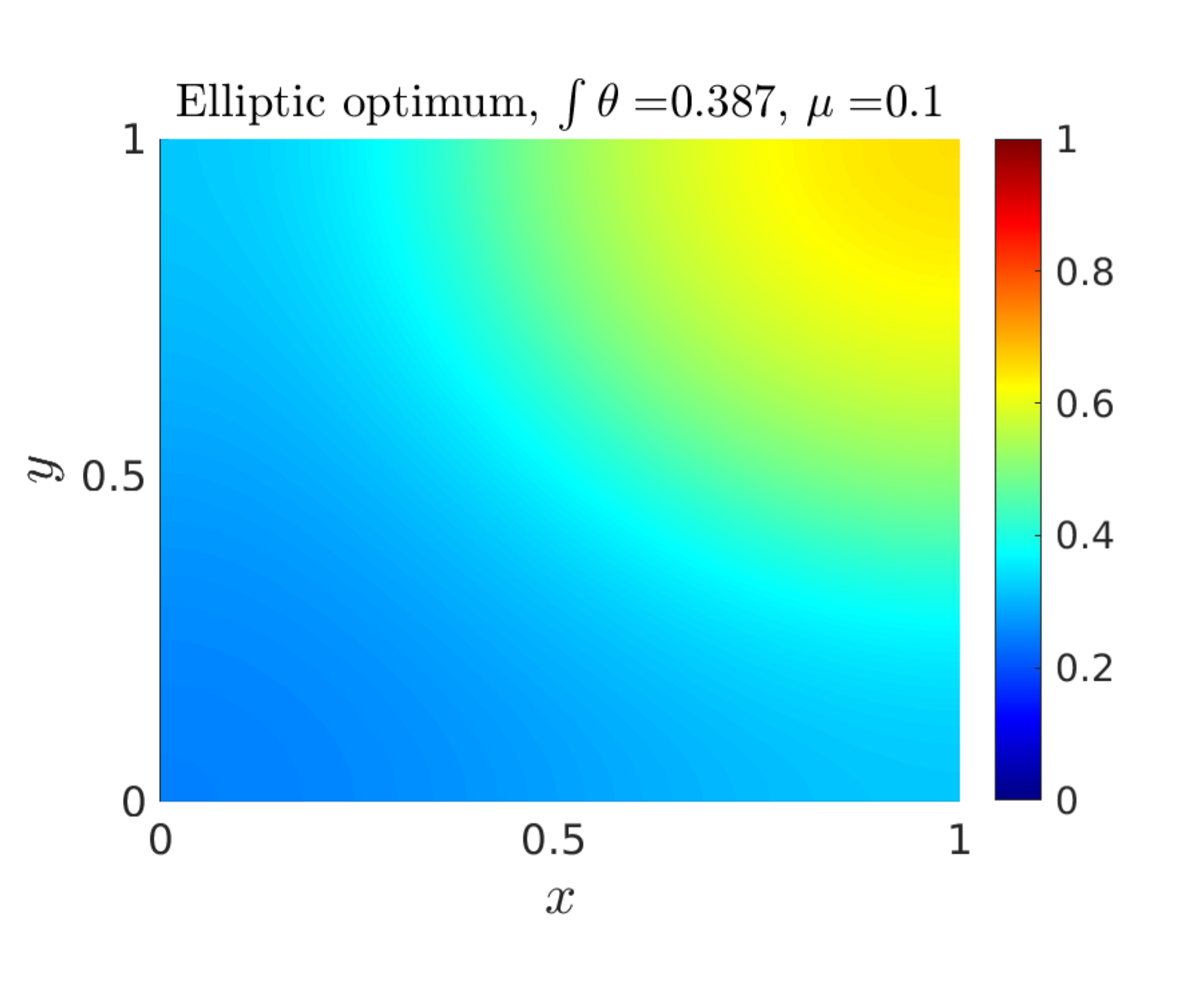}
\caption{}%{For a 'large' diffusivity, we obtain a very concentrated optimal resources distribution.}
\end{center}
\end{figure}

\begin{figure}[h!]
\begin{center}
\includegraphics[width=6cm]{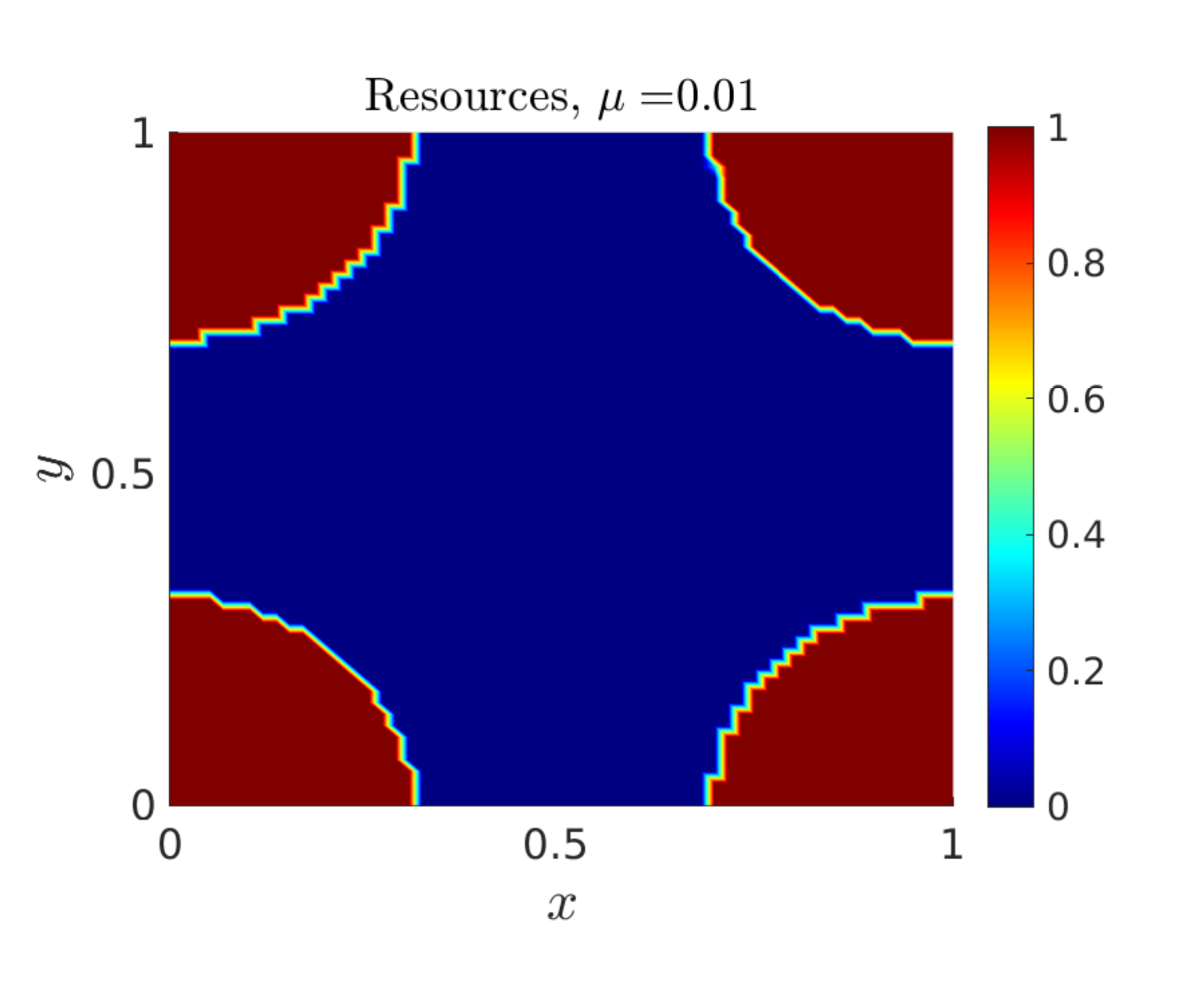}\hspace{1cm}
\includegraphics[width=6cm]{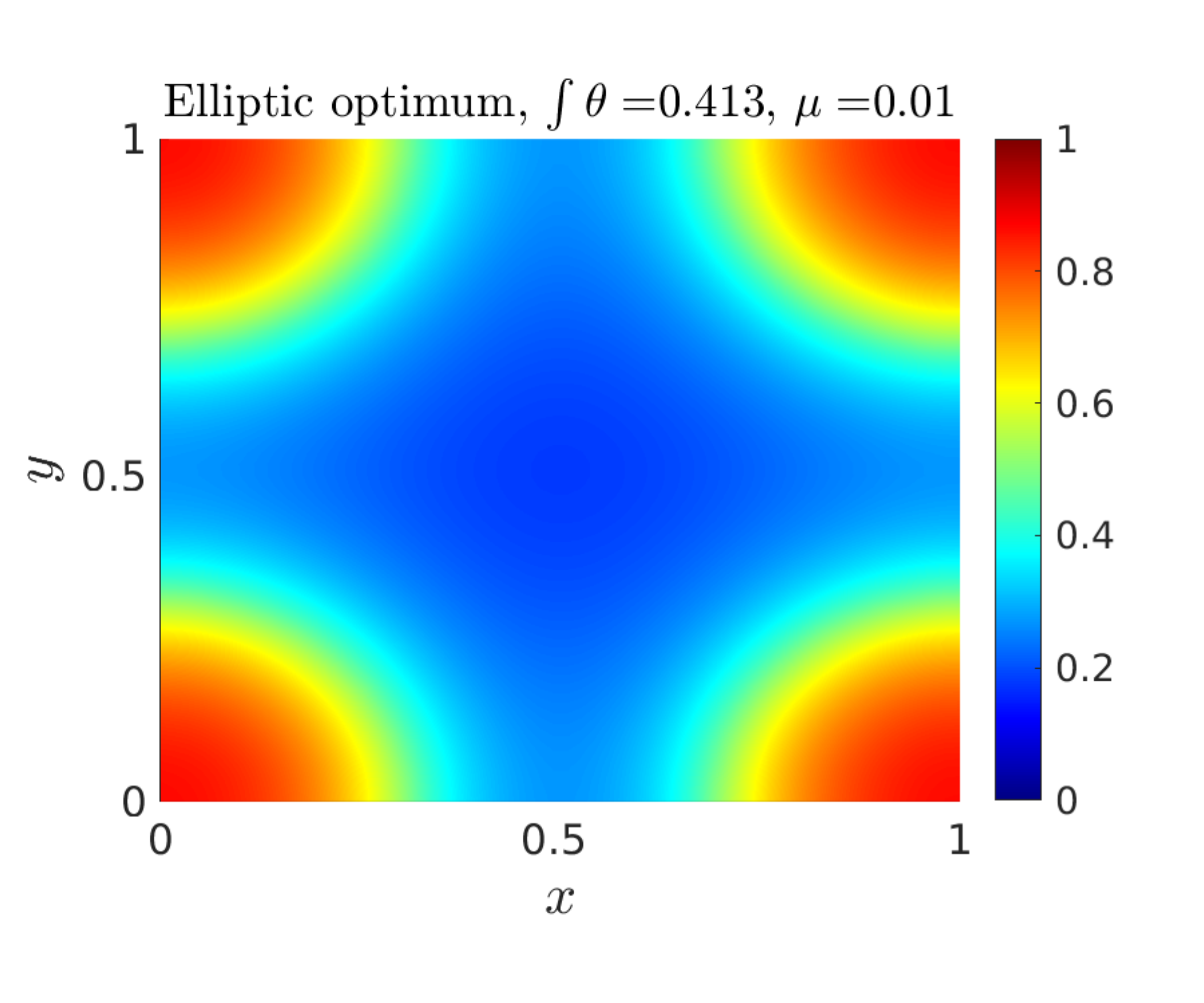}
\caption{}%{For a 'large' diffusivity, we obtain a very concentrated optimal resources distribution.}
\end{center}
\end{figure}

\begin{figure}[h!]
\begin{center}
\includegraphics[width=6cm]{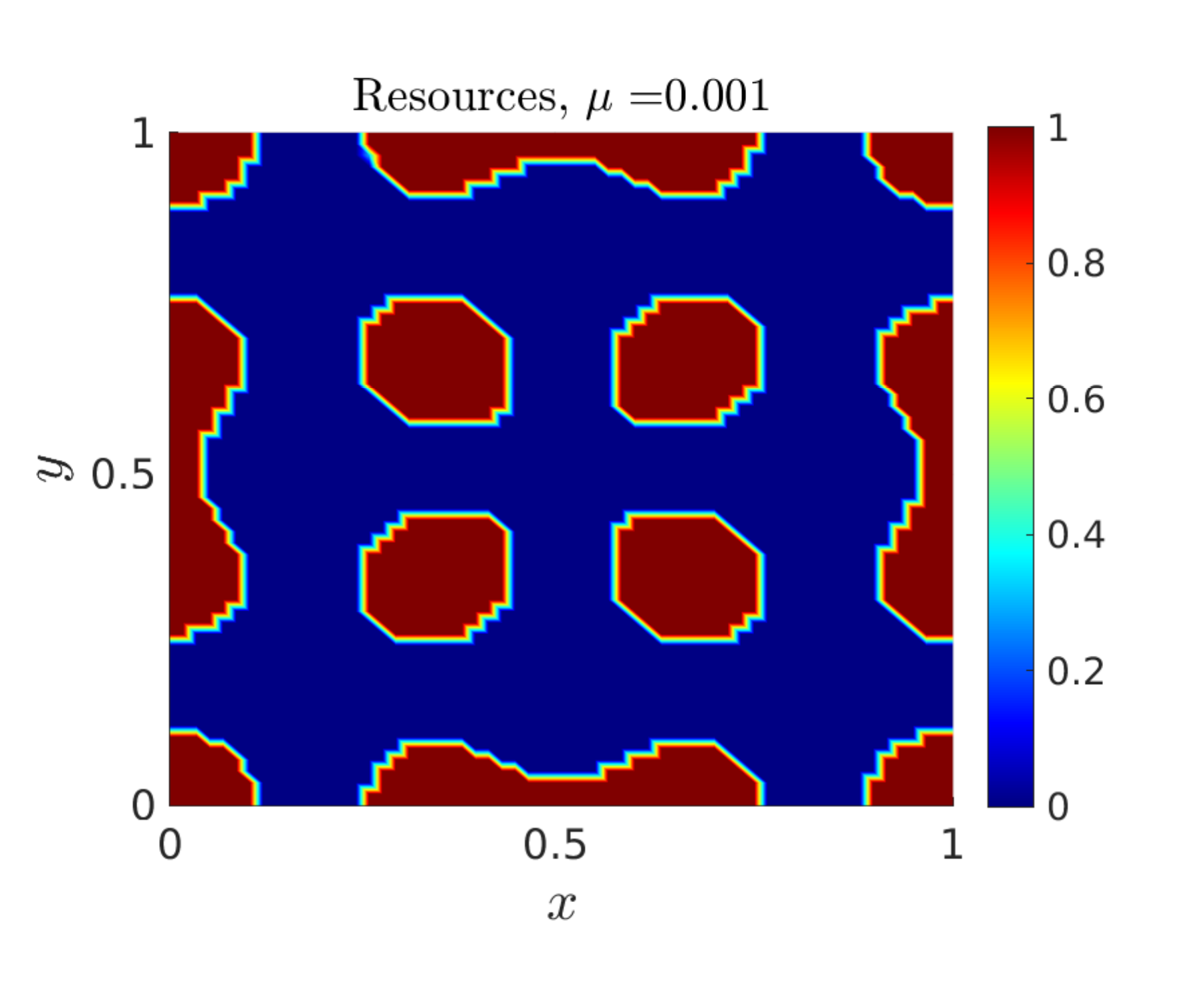}\hspace{1cm}
\includegraphics[width=6cm]{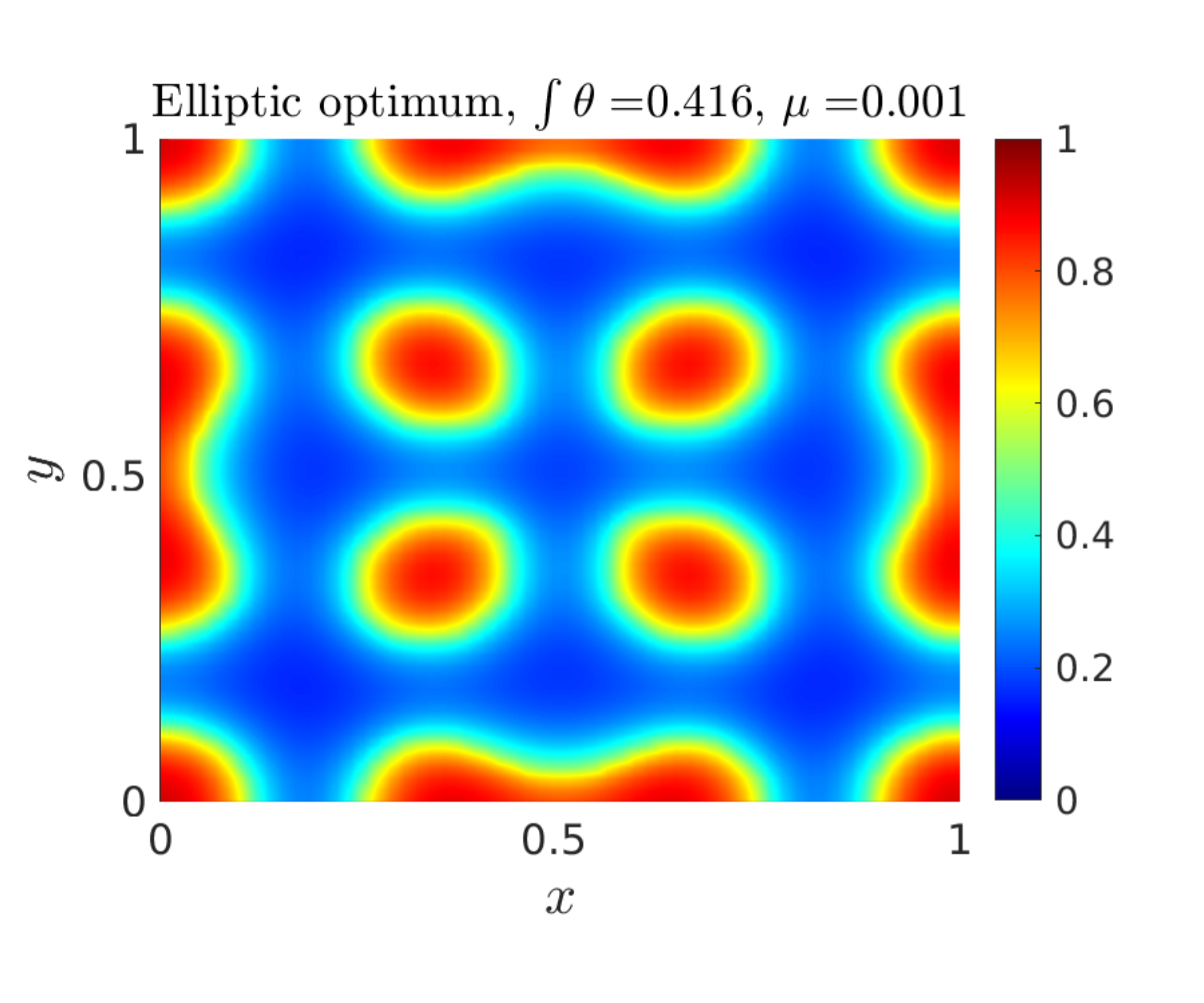}
\caption{}%{For a 'large' diffusivity, we obtain a very concentrated optimal resources distribution.}
\end{center}
\end{figure}

\begin{figure}[h!]
\begin{center}
\includegraphics[width=6cm]{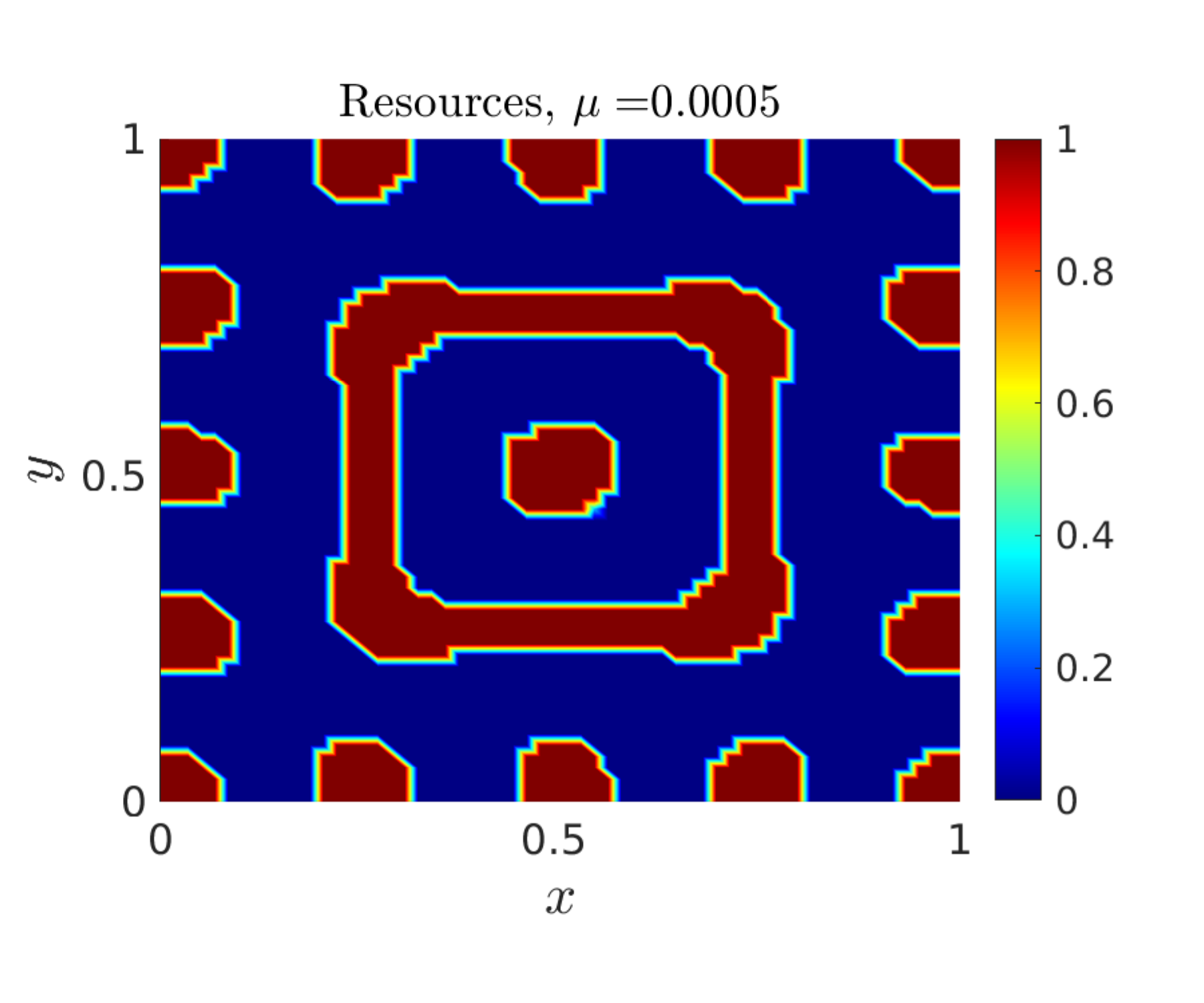}\hspace{1cm}
\includegraphics[width=6cm]{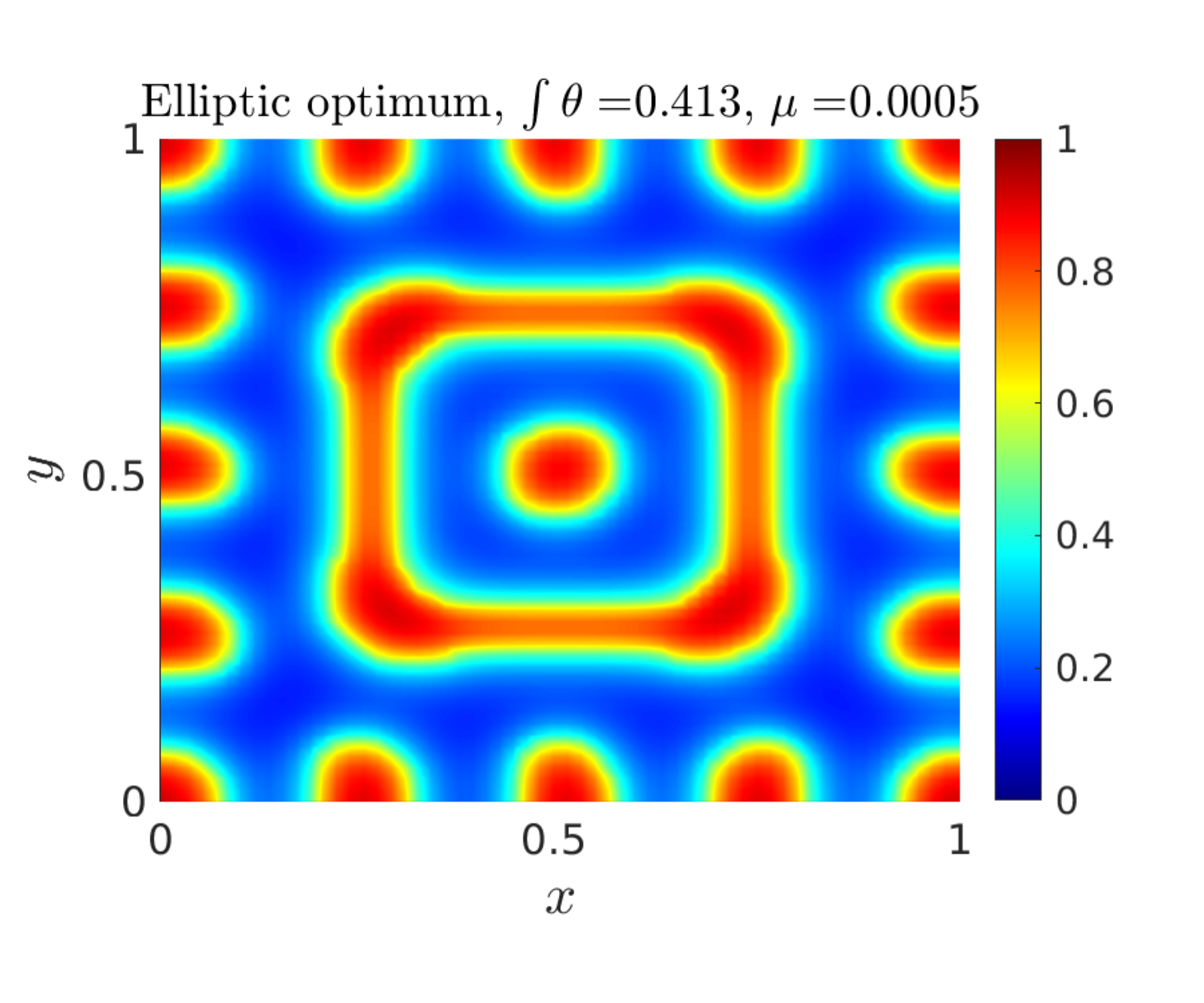}
\caption{}%{For a 'large' diffusivity, we obtain a very concentrated optimal resources distribution.}
\end{center}
\end{figure}

\begin{figure}[h!]
\begin{center}
\includegraphics[width=6cm]{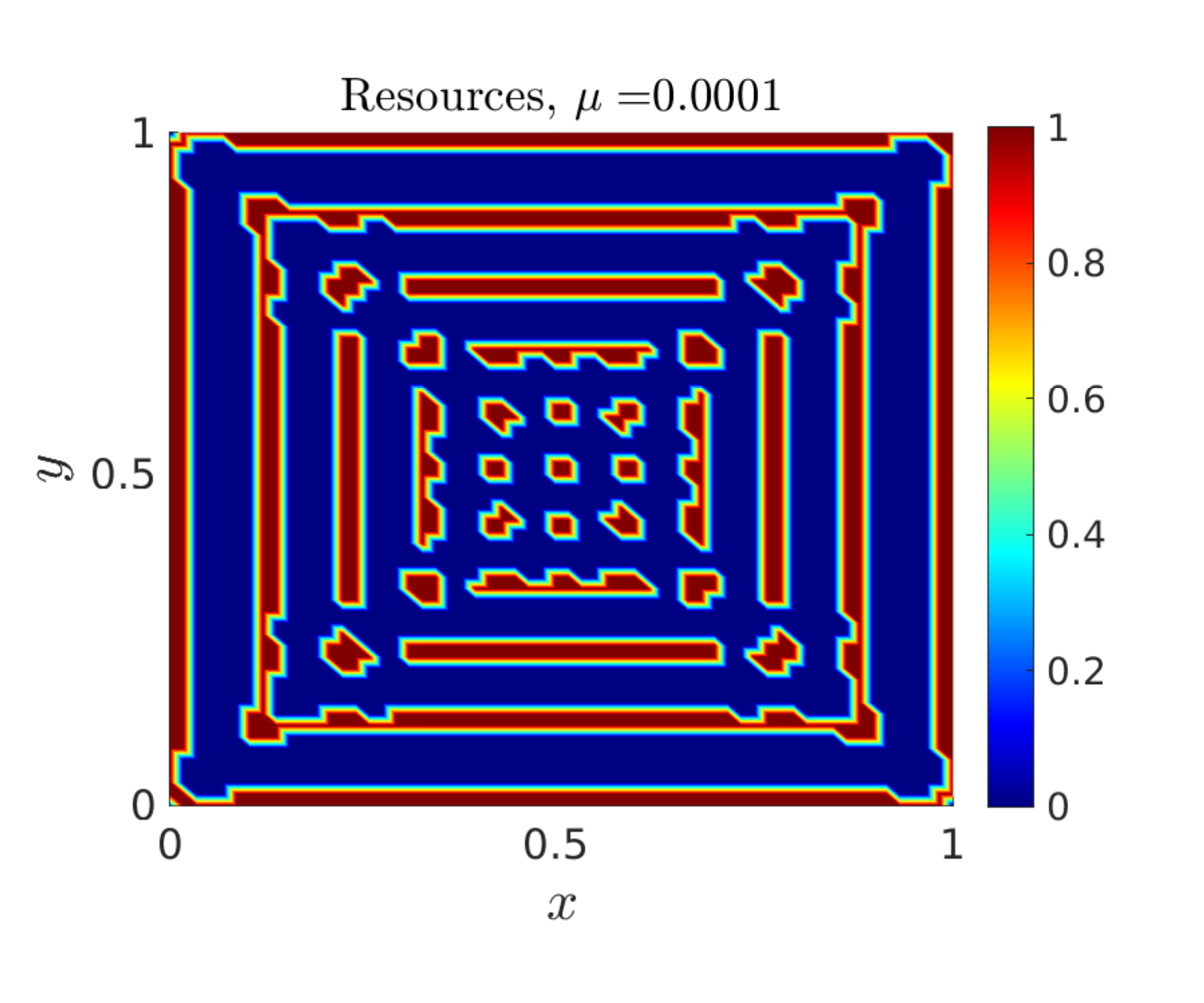}\hspace{1cm}
\includegraphics[width=6cm]{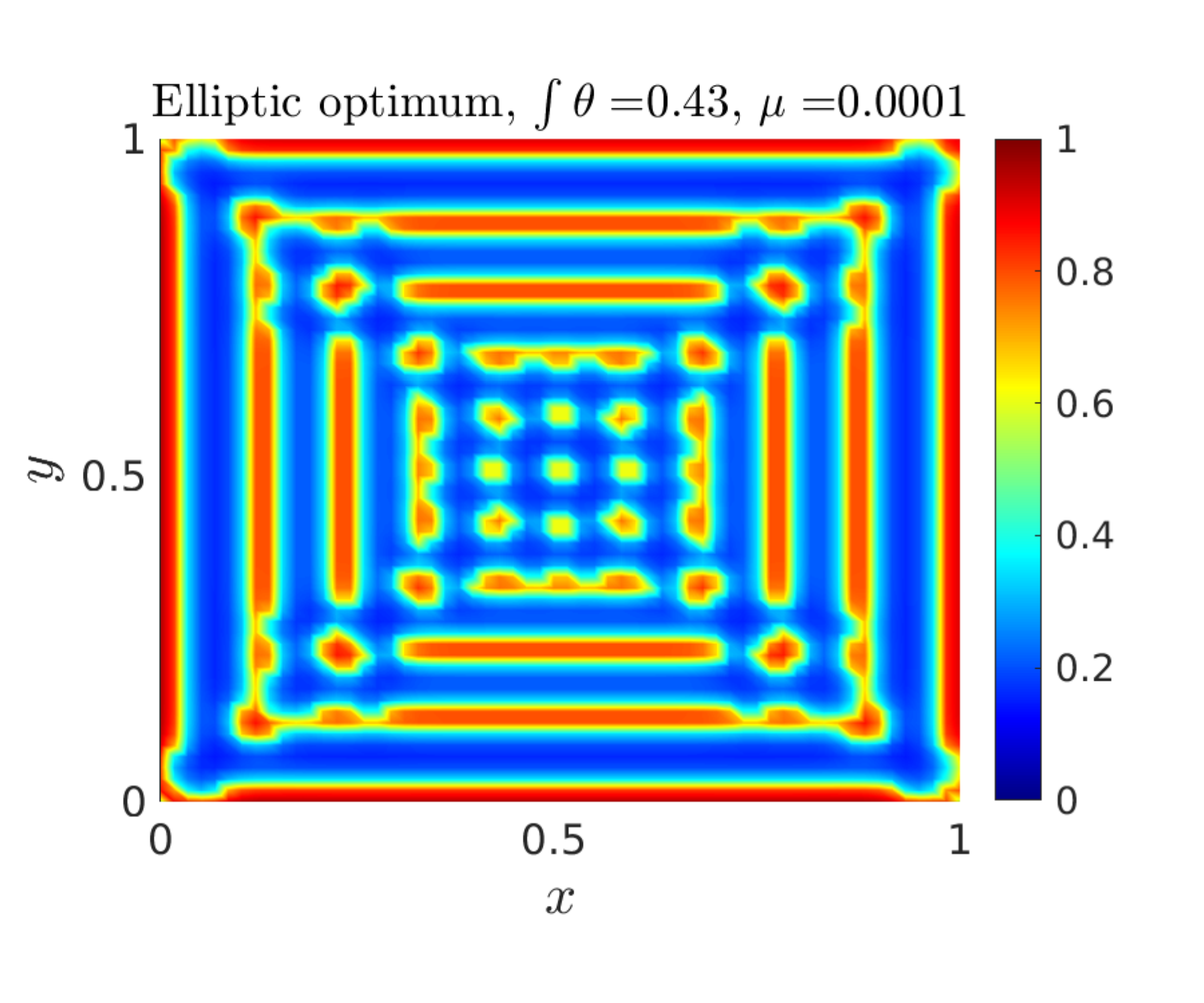}
\caption{}%{For a 'large' diffusivity, we obtain a very concentrated optimal resources distribution.}
\end{center}
\end{figure}

\clearpage
\newpage

\subsubsection{$\kappa=1$, $m_0=0.6$} 

\begin{figure}[h!]
\begin{center}
\includegraphics[width=6cm]{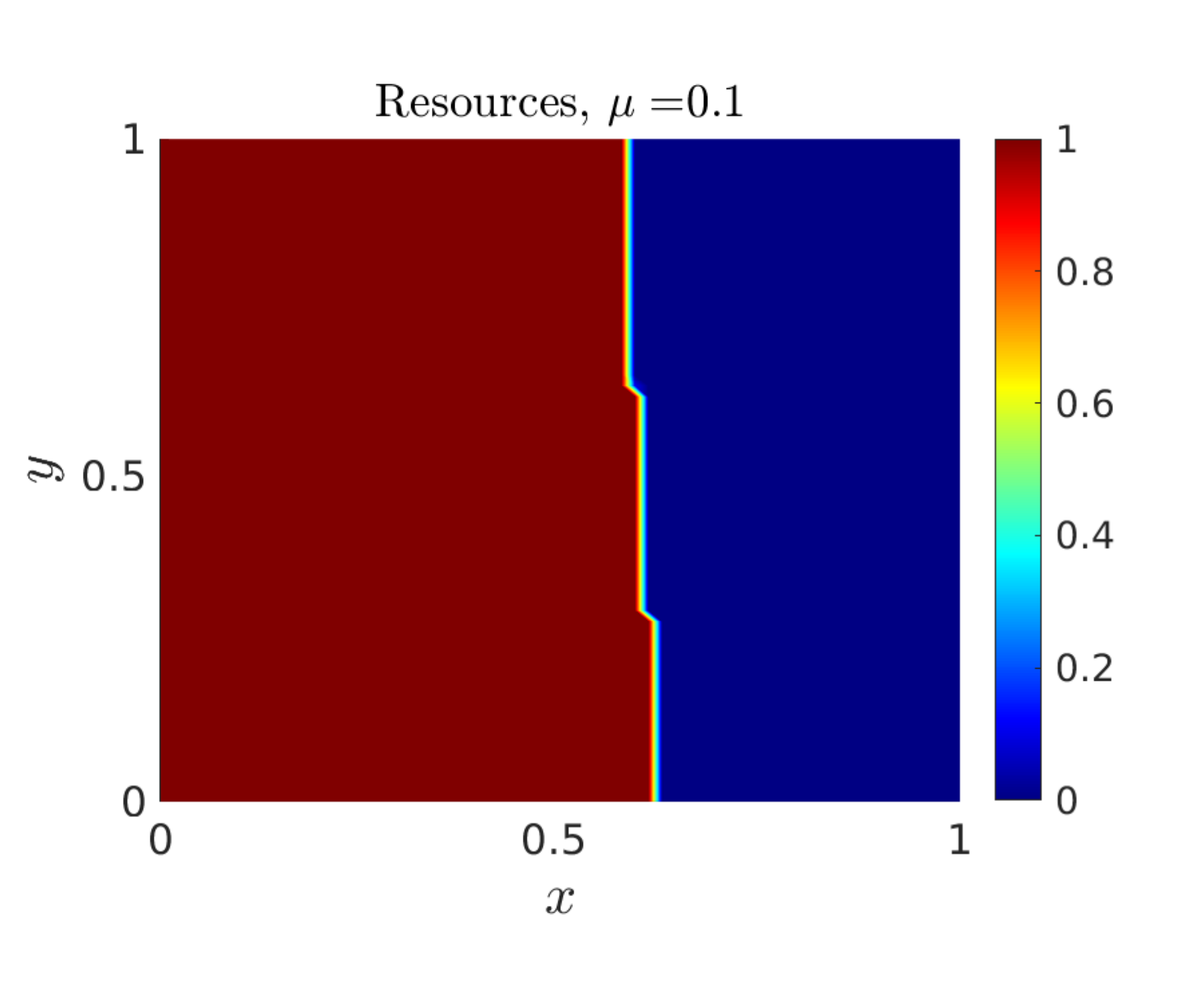}\hspace{1cm}
\includegraphics[width=6cm]{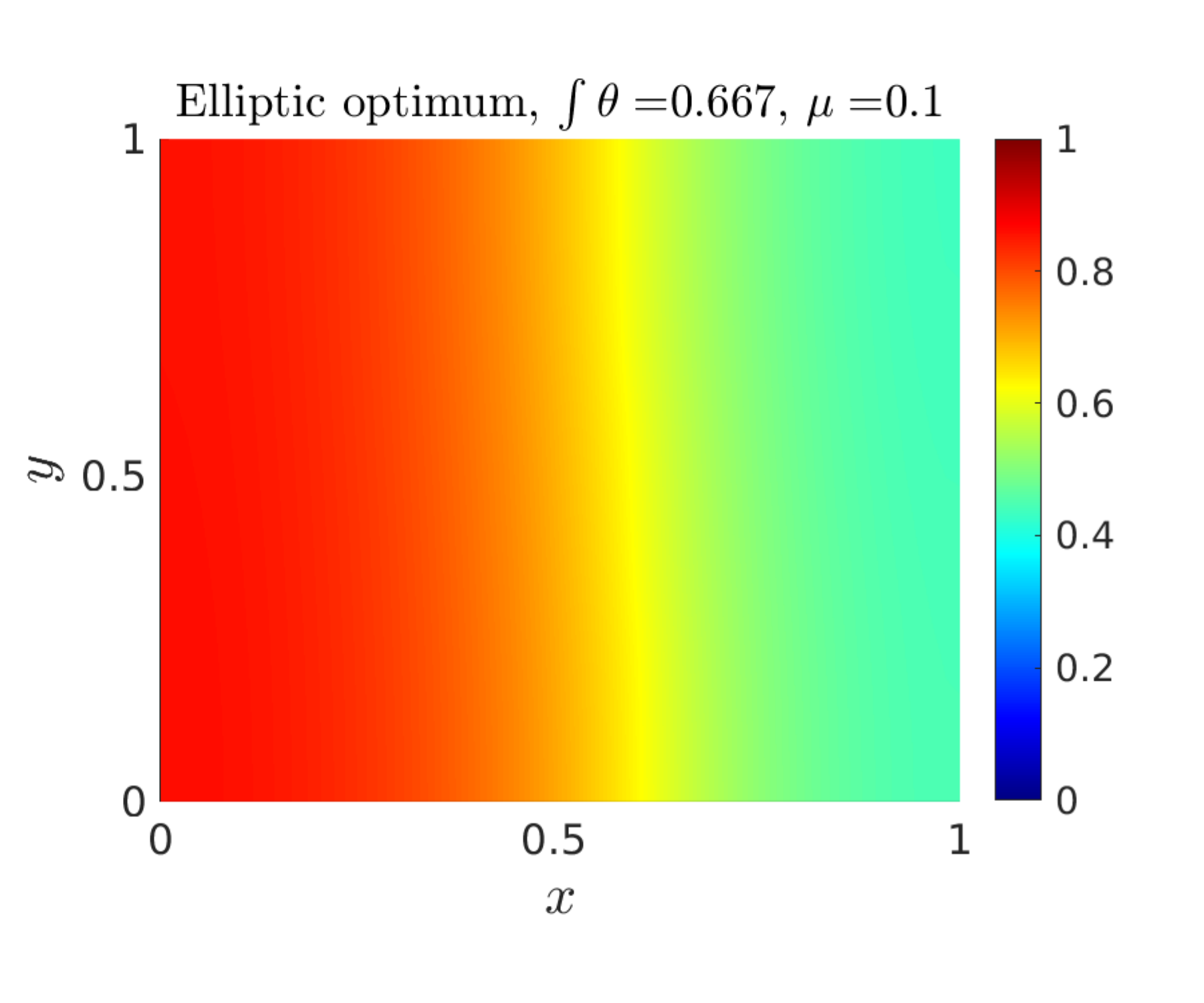}
\caption{}%{For a 'large' diffusivity, we obtain a very concentrated optimal resources distribution.}
\end{center}
\end{figure}

\begin{figure}[h!]
\begin{center}
\includegraphics[width=6cm]{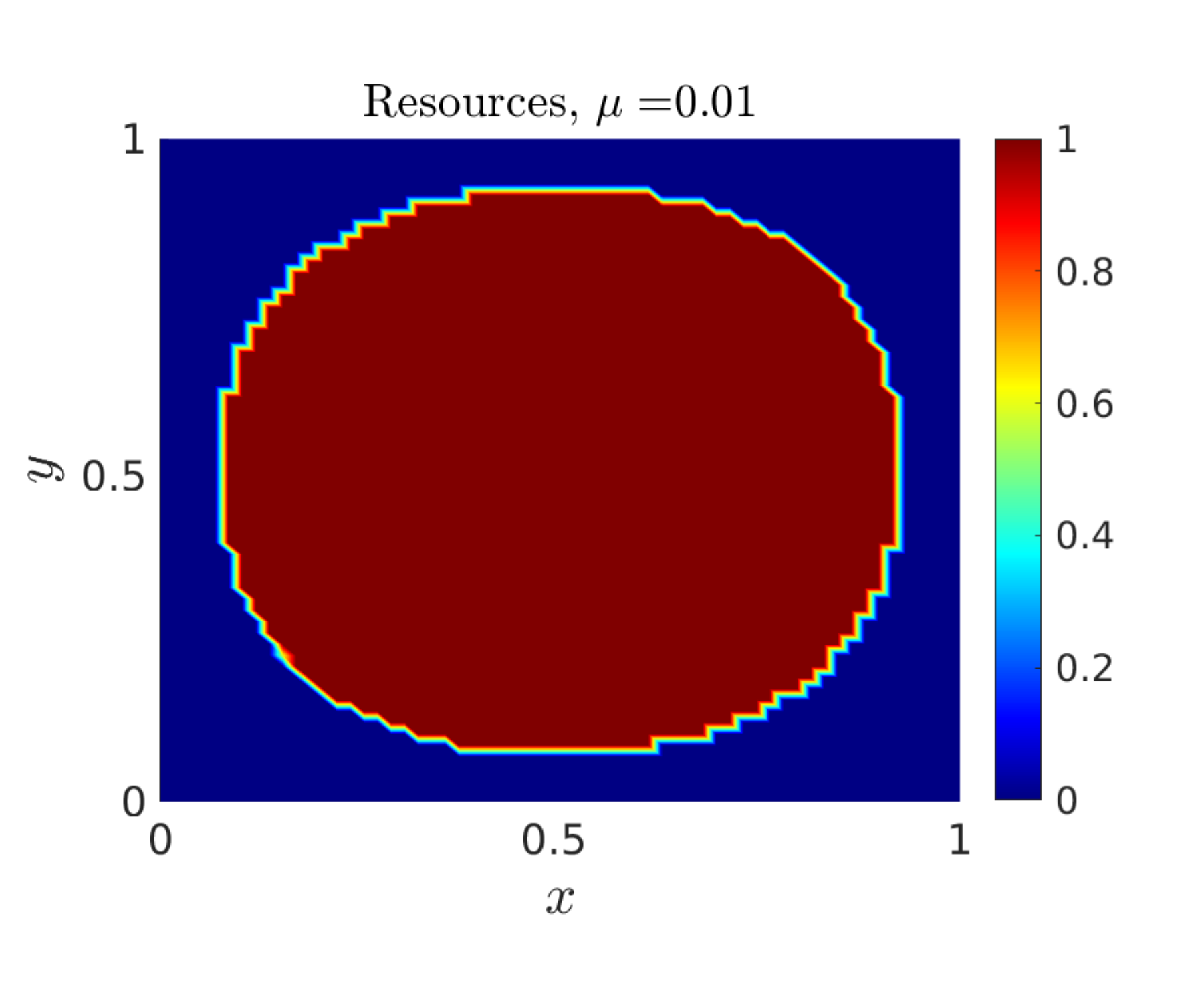}\hspace{1cm}
\includegraphics[width=6cm]{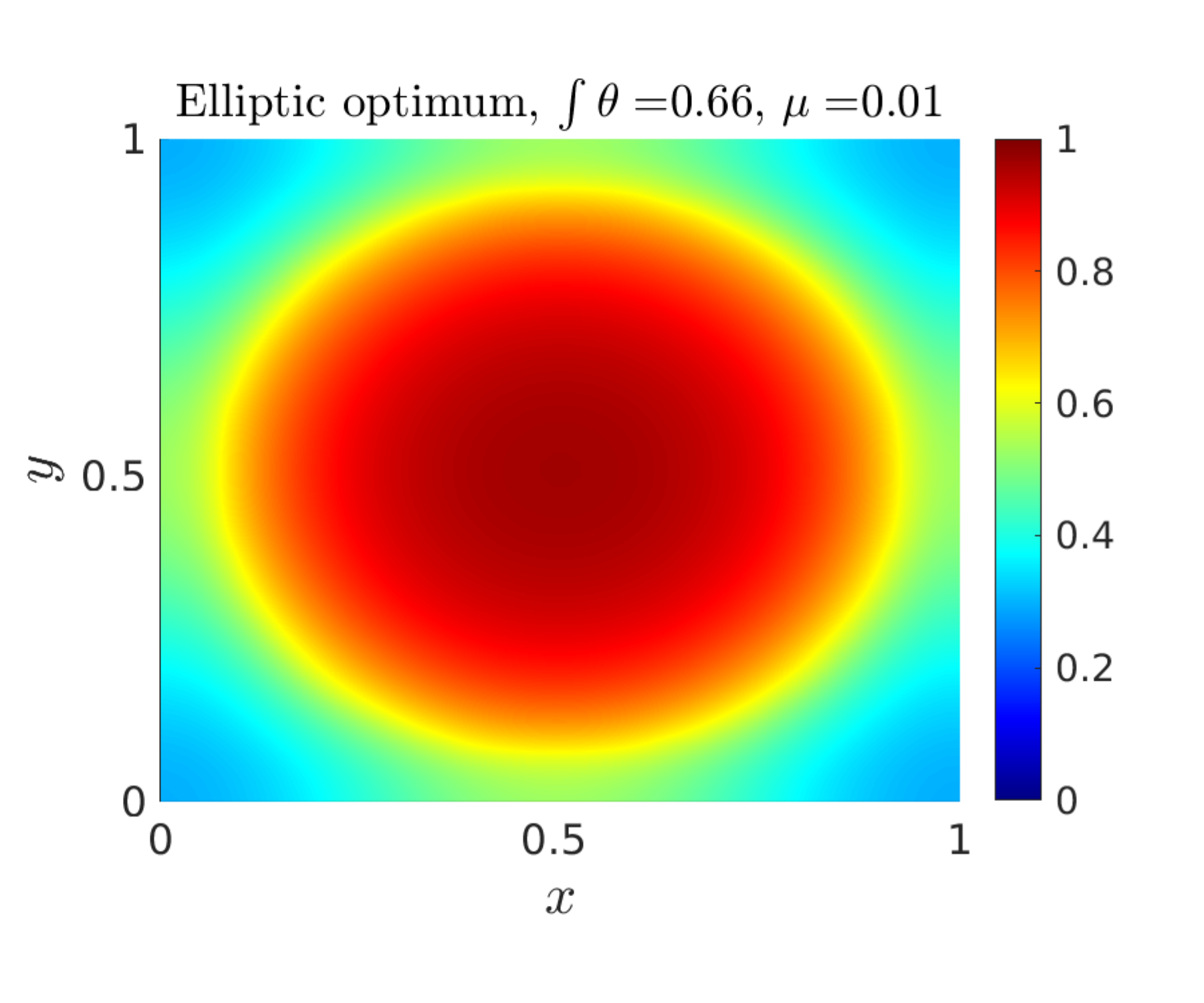}
\caption{}%{For a 'large' diffusivity, we obtain a very concentrated optimal resources distribution.}
\end{center}
\end{figure}

\begin{figure}[h!]
\begin{center}
\includegraphics[width=6cm]{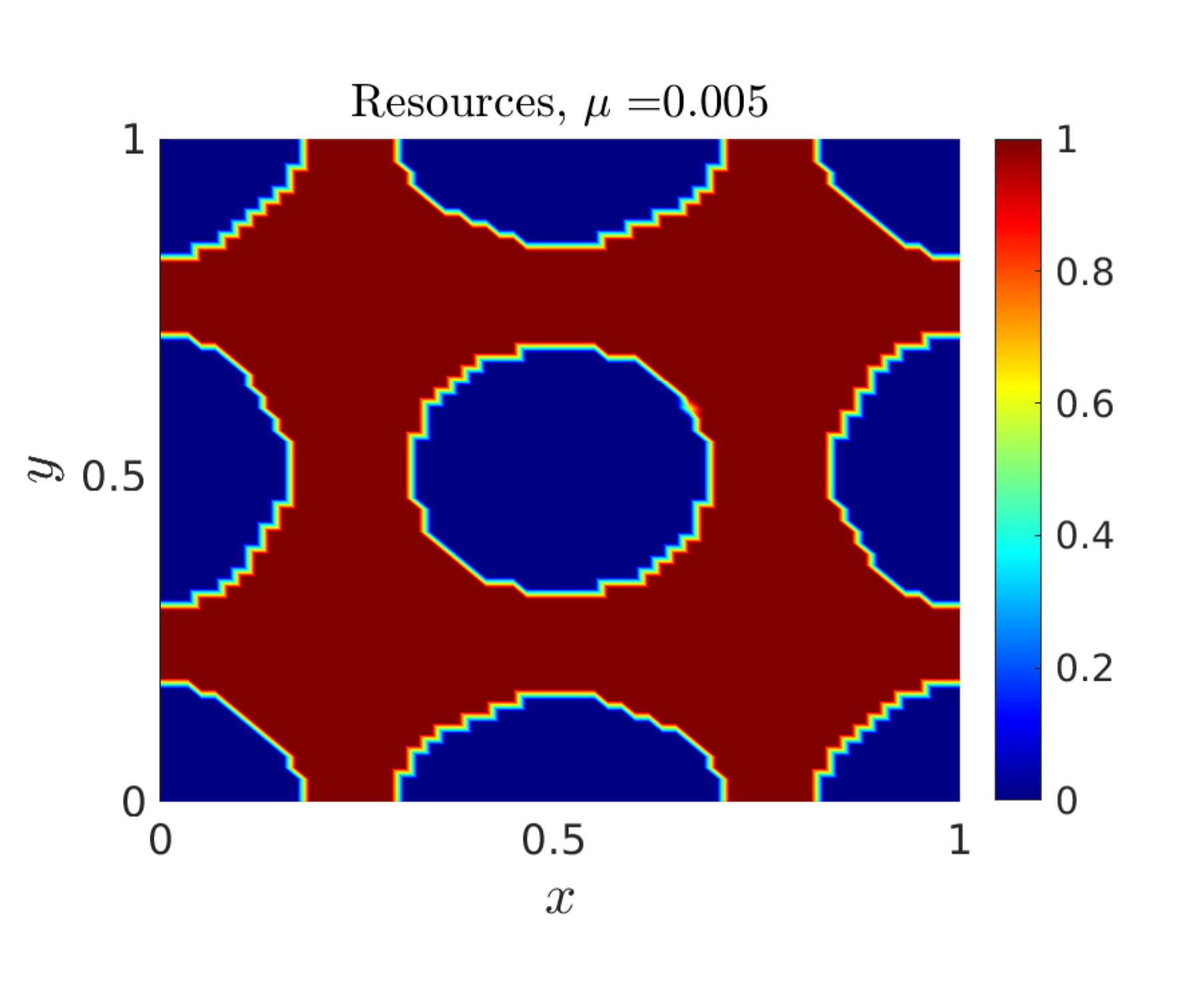}\hspace{1cm}
\includegraphics[width=6cm]{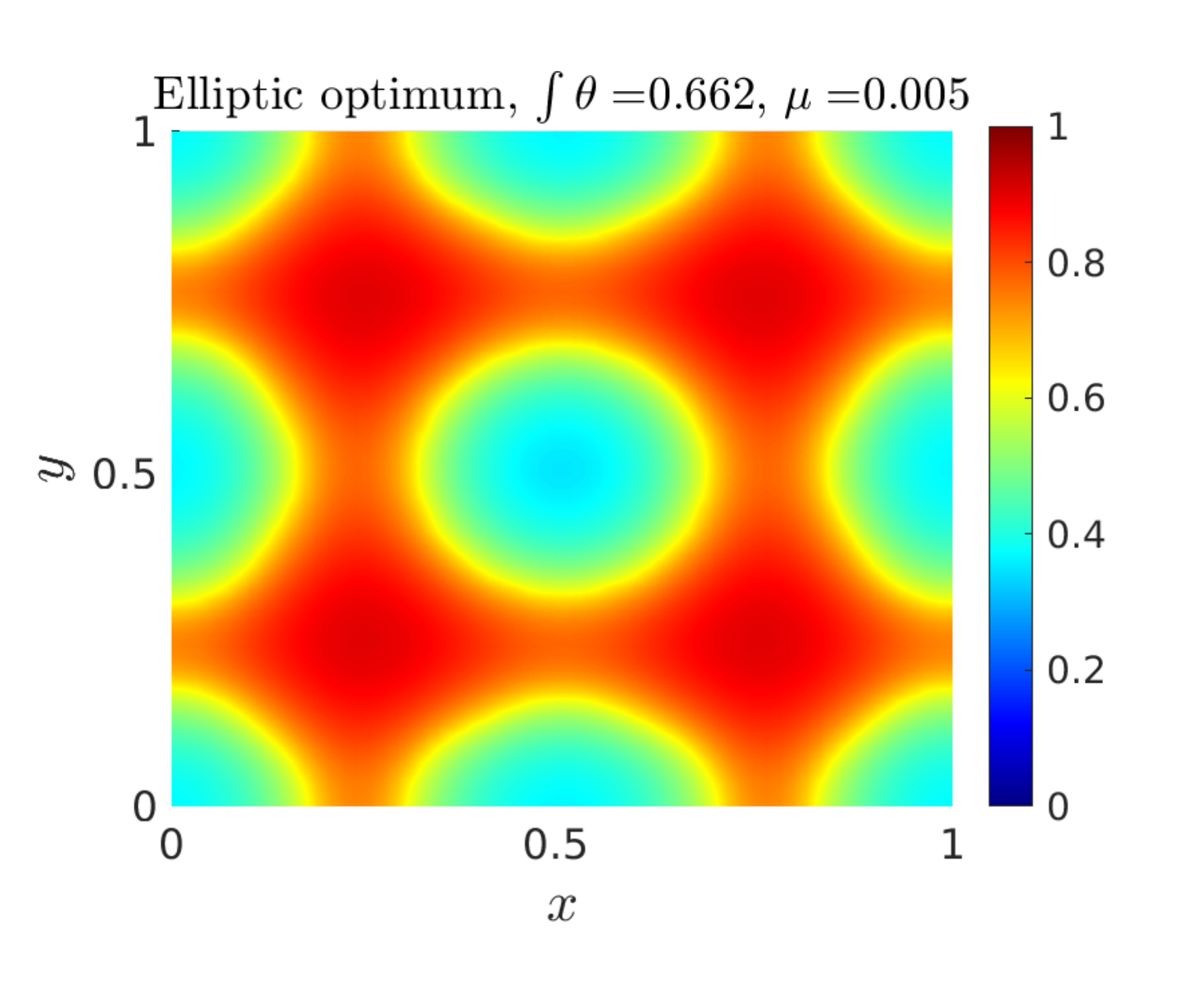}
\caption{}%{For a 'large' diffusivity, we obtain a very concentrated optimal resources distribution.}
\end{center}
\end{figure}

\begin{figure}[h!]
\begin{center}
\includegraphics[width=6cm]{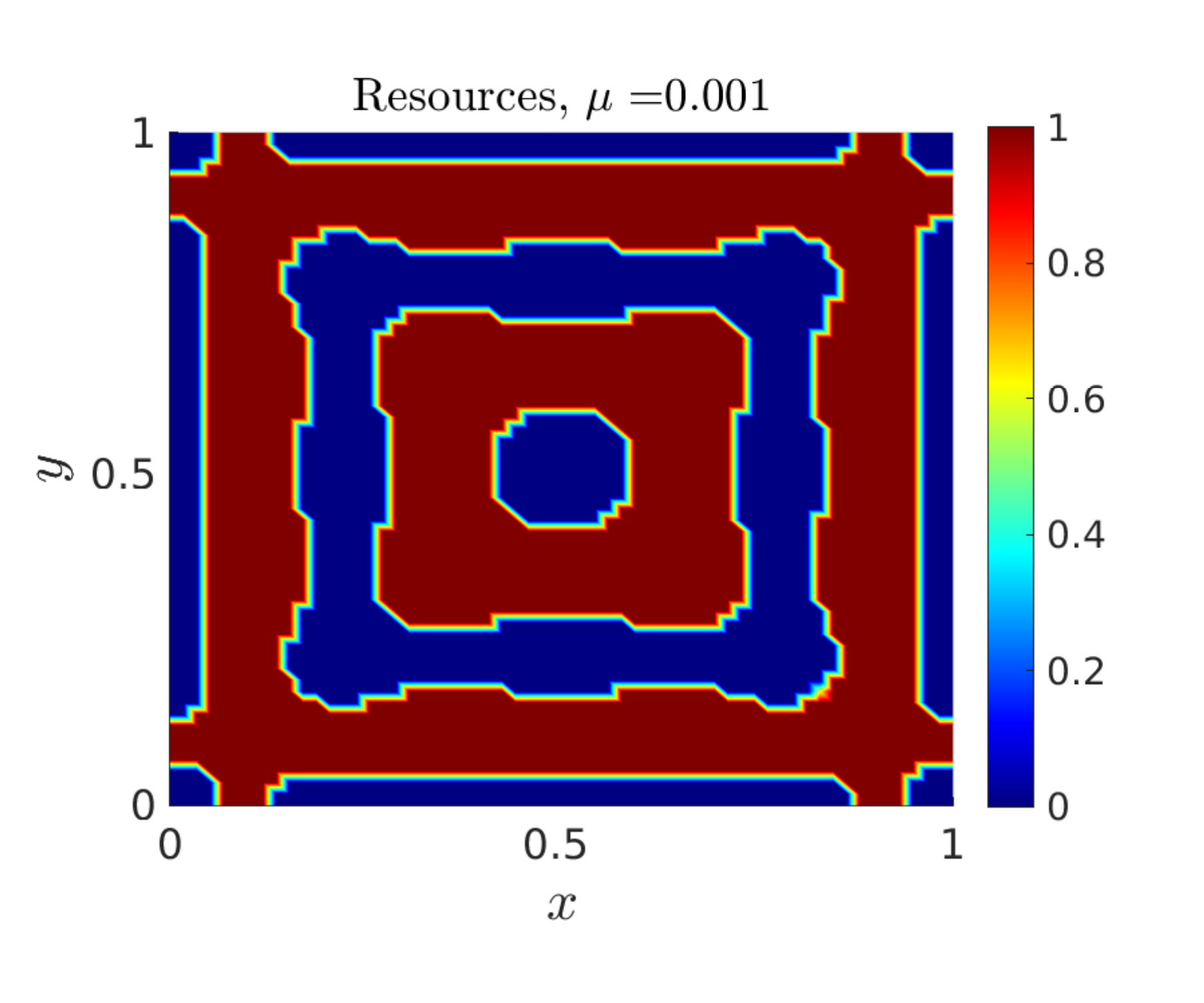}\hspace{1cm}
\includegraphics[width=6cm]{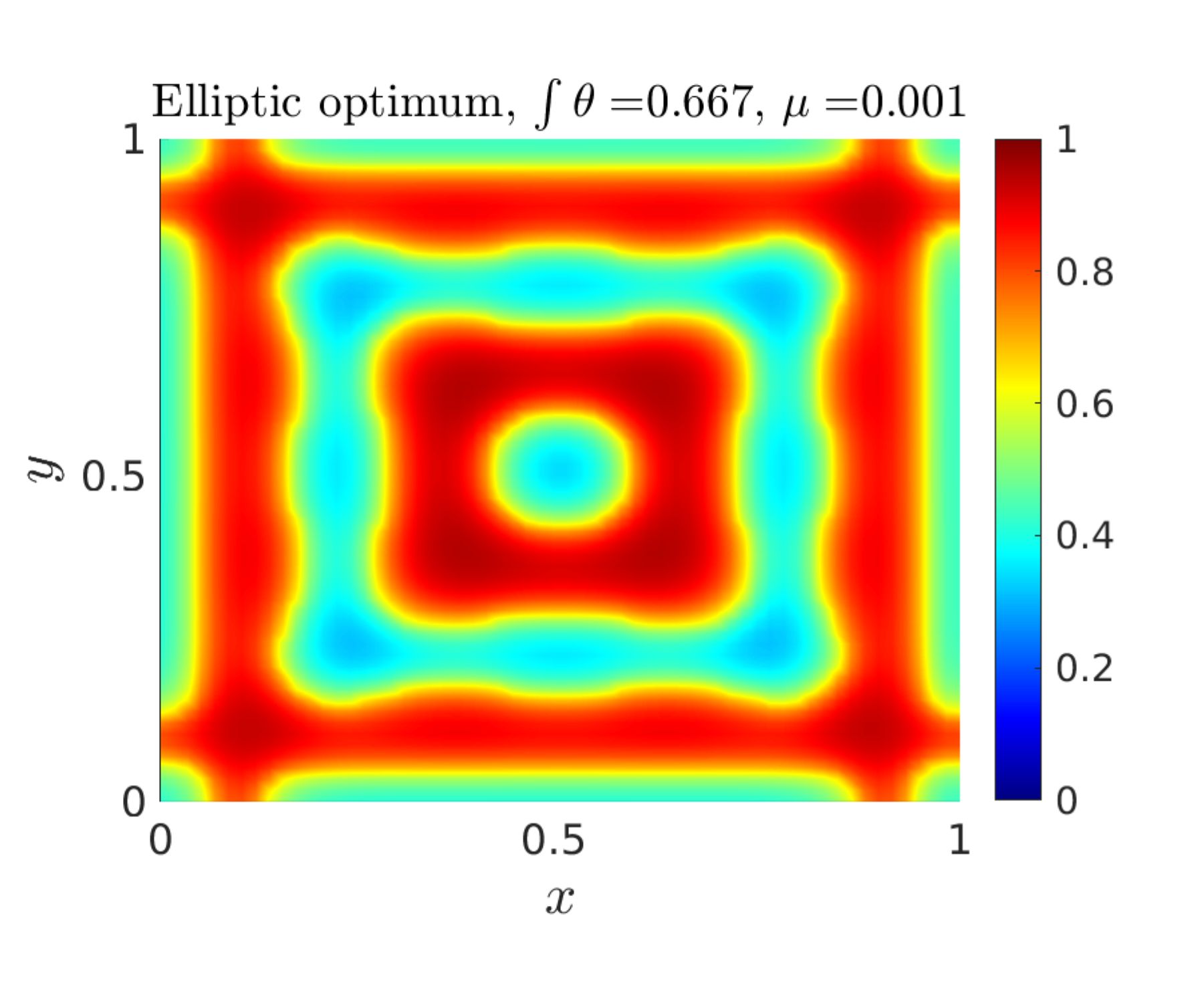}
\caption{}%{For a 'large' diffusivity, we obtain a very concentrated optimal resources distribution.}
\end{center}
\end{figure}

\begin{figure}[h!]
\begin{center}
\includegraphics[width=6cm]{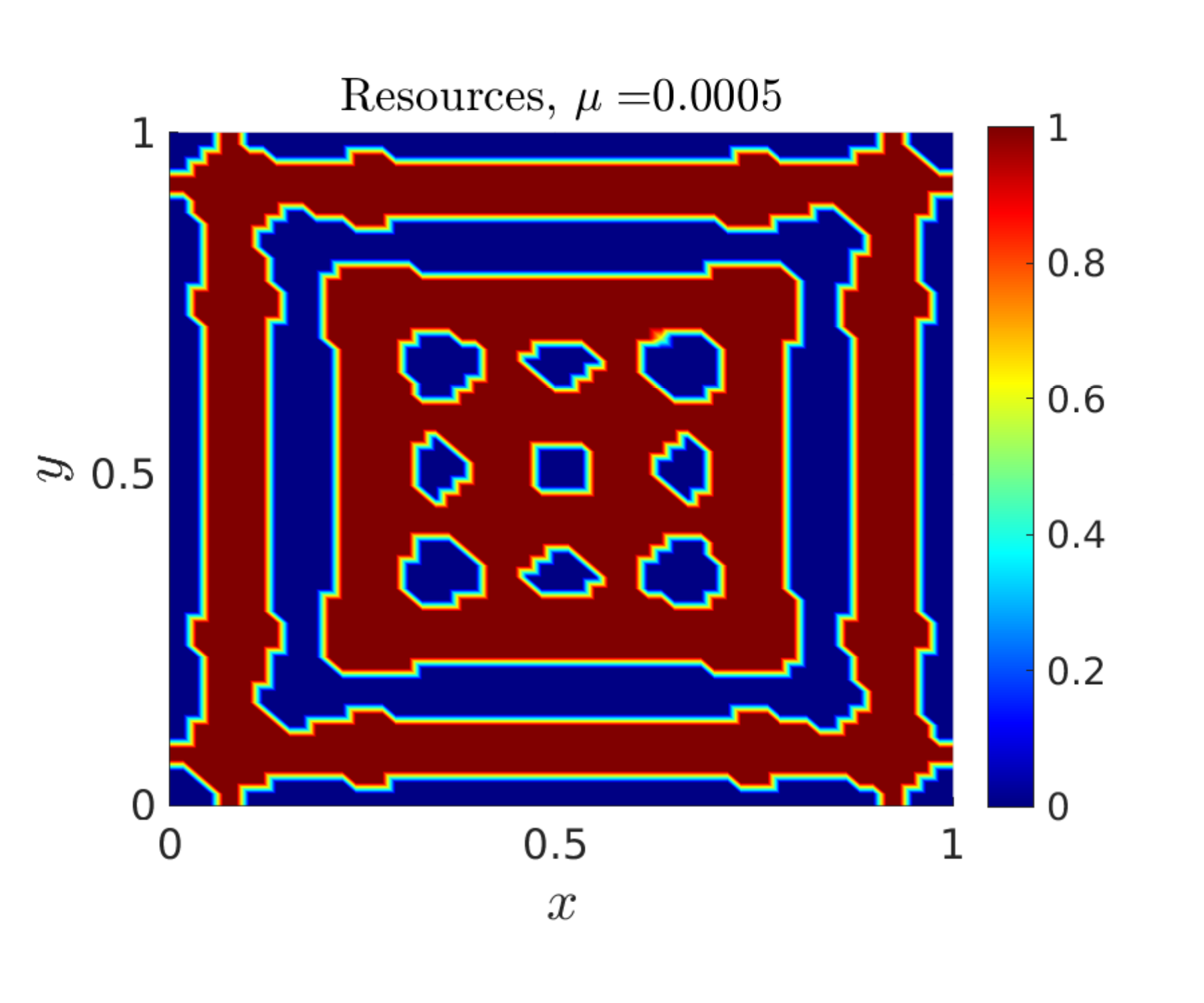}\hspace{1cm}
\includegraphics[width=6cm]{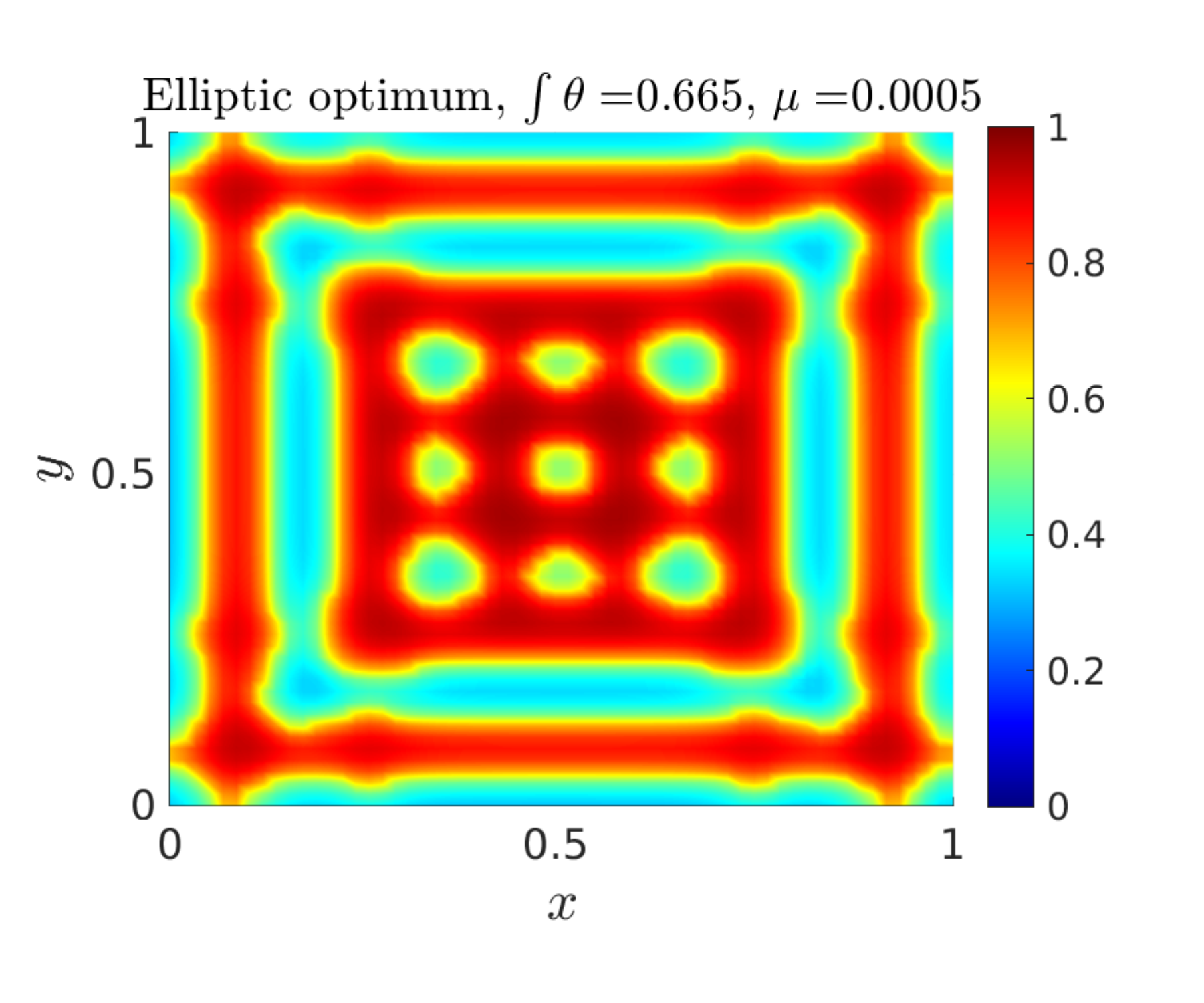}
\caption{}%{For a 'large' diffusivity, we obtain a very concentrated optimal resources distribution.}
\end{center}
\end{figure}

%\begin{figure}[h!]
%\begin{center}
%\includegraphics[width=6cm]{../06d2Ellm10-1.png}\hspace{1cm}
%\includegraphics[width=6cm]{../06d2EllState10-1.png}
%{For a 'large' diffusivity, we obtain a very concentrated optimal resources distribution.}
%\end{center}
%\end{figure}

%\clearpage

%\newpage
\clearpage

\newpage
 
\section{Conclusion and open questions}
In this article, the property we have obtained for the solutions of \eqref{PV} in the case of rectangular geometries underlines the complexity of this variational problem. Several fundamental questions still remain open:
\begin{itemize}
\item \underline{The bang-bang property for general diffusivities:} in other words, can we prove that any solution (or at least one) of \eqref{PV} is a characteristic function for any $\mu>0$? This conjecture was raised in \cite{Ding2010} and, as mentioned in the Introduction, has received partial answers \cite{MNP,NagaharaYanagida} that seem to point towards a positive answer, as the numerical simulations presented in this article do.
\item \underline{Qualitative properties of optimizers for general domains:} is it true that the fragmentation property of Theorem \ref{Th:Frag} holds in more general domain? At this point, we can give no conclusive answer. In the proof of Theorem \ref{Th:Frag}, the crucial element provided by the rectangular geometry is Estimate \ref{Eq:L2}, which bounds  $\underset{\mu \to 0}{\underline\lim}\left(\sup_{\mathcal M(\O)} F(\cdot,\mu)\right)$ from below by $m_0+\eta$ for some $\eta>0$. The proof relies on explicit constructions, and we do not know whether or not other types of arguments could lead to such an estimate.

It should be noted that another question on the geometry of optimal resources distributions was asked in \cite{Ding2010}: when the parameter $m_0$ is small and the domain $\O$ is curved, is it better to concentrate resources near the curved parts of the boundary? We believe this problem to be highly challenging given that, for the problem of the optimal survival ability with Neumann boundary conditions, for which many qualitative results have been established \cite{KaoLouYanagida,LamboleyLaurainNadinPrivat}, the same kind of questions (such as: is it true that for general domains the optimal resources distribution for the survival ability touches the boundary?) have not yet received complete mathematical answers.

\item\underline{Behaviour of the maximizers when $\mu \to 0$:} the question here would be to understand the behaviour of sequences of maximizers as $\mu\to 0$ in the one-dimensional case $\O=(0;1)$. As mentioned in Remark \ref{Re:NoPeriod}, an interesting question is to know whether or not solutions of \eqref{PV} exhibit a periodic structure. This seems very challenging. Another weaker qualitative property of maximizers would be given by the answer to the following question: is it true that, for any sequence $\{\mu_k\}_{k\in \N}$ converging to 0 as $k\to \infty$, the sequence of maximizers of $(\color{Plum}P_{\mu_k} )$ converges weakly to a constant?  
\end{itemize}
\color{black}
\paragraph{Acknowledgment.}
The authors wish to warmly thank the referees for their numerous comments and remarks. They would also like to thank K. Nagahara for scientific exchanges.

This work was started during a research stay of the authors at the the Friedrich-Alexander-Universit\"{a}t at the invitation of Enrique Zuazua and the authors  wish to thank E. Zuazua and the FA{U} University for their hospitality.

I. Mazari was  supported by the French ANR Project ANR-18-CE40-0013 - SHAPO on Shape Optimization and by the Austrian Science Fund (FWF) through the grant I4052-N32 .

D. Ruiz-Balet was  supported by the European Research Council (ERC) under the European Union's Horizon 2020 research and innovation programme (grant agreement No. 694126-DyCon).

This project has received funding from the European Union's Horizon 2020 research and innovation programme under the Marie Sklodowska-Curie grant agreement No.765579-ConFlex, grant MTM2017-92996 of MINECO (Spain), ICON of the French ANR and "Nonlocal PDEs: Analysis, Control and Beyond", AFOSR Grant FA9550-18-1-0242 and the Alexander von Humboldt-Professorship program.

\bibliographystyle{abbrv}
%\nocite{*}

\bibliography{BiblioFrag}

\end{document}